\documentclass[reqno]{amsart}

\usepackage{amsfonts, amssymb, amsmath}
\title{Measurable Differentiable Structures on Doubling Metric Spaces}
\date{14 August 2012 (last modified)}

\author{Jasun Gong}
\address{Jasun Gong \hfill\break\indent
Institute of Mathematics \hfill\break\indent
Aalto University \hfill\break\indent 
P.O. Box 11100  \hfill\break\indent  
FI-00076 Aalto
\hfill\break\indent Finland
}
\email{jasun.gong@aalto.fi}

\theoremstyle{plain}
\newtheorem{thm}{Theorem}[section]

\newtheorem{cor}[thm]{Corollary}
\newtheorem{lemma}[thm]{Lemma}
\newtheorem{ques}[thm]{Question}

\theoremstyle{definition}
\newtheorem{defn}[thm]{Definition}
\newtheorem{rmk}[thm]{Remark}

\theoremstyle{remark}

\numberwithin{equation}{section}

\renewcommand{\a}{\alpha}
\renewcommand{\b}{\beta}

\renewcommand{\d}{\delta}
\newcommand{\diam}{\operatorname{diam}}
\newcommand{\e}{\epsilon}
\newcommand{\f}{\mathfrak{f}}

\newcommand{\g}{\mathfrak{g}}
\newcommand{\h}{\mathfrak{h}}
\renewcommand{\H}{\mathcal{H}}

\newcommand{\Lip}{\operatorname{Lip}}
\newcommand{\N}{\mathbb{N}}

\newcommand{\p}{\mathfrak{p}}

\newcommand{\R}{\mathbb{R}}
\newcommand{\U}{\Upsilon}
\newcommand{\wsto}{\stackrel{*}{\rightharpoonup}}

\begin{document}
\maketitle

\begin{abstract}
On metric spaces equipped with doubling measures, we prove that a differentiability theorem holds for Lipschitz functions if and only if the space supports nontrivial (metric) derivations in the sense of Weaver \cite{WeaverED} that satisfy an additional infinitesmal condition.  In particular it extends the case of spaces supporting Poincar\'e inequalities, as first proven by Cheeger \cite{Cheeger}, as well as the case of spaces satisfying the Lip-lip condition of Keith \cite{Keith}.

The proof relies on generalised ``change of variable'' arguments that are made possible by the linear algebraic structure of derivations.  As a crucial step in the argument, we also prove new rank bounds for derivations with respect to doubling measures.

\end{abstract}

\section{Introduction}

In 1919 Rademacher 
proved that Lipschitz functions on $\R^n$ are a.e.\ differentiable with respect to the Lebesgue measure.  Since then, many mathematicians have pursued similar differentiability results in increasingly general settings.  The main result of this note follows this same direction but in the context of metric spaces equipped with Borel regular measures, or {\em metric measure spaces}.

Before proceeding to the theorem itself, it is worth recalling the geometric considerations that led to this general framework.  Pansu \cite{Pansu} was motivated by the Mostow rigidity phenomenon for negatively-curved manifolds and their ideal boundaries.  To this end, he showed that a Rademacher-type theorem holds true for Carnot groups \cite{Gromov:CC}, \cite{Bellaiche}, i.e.\ certain nilpotent Lie groups with similar metric structures as these ideal boundaries.
Heinonen and Koskela  \cite{Heinonen:Koskela} further identified a general class of metric measure spaces and developed on them a rich theory of quasi-conformal mappings, a key tool in the geometry of hyperbolic manifolds.  These spaces are determined by two properties: (1) the doubling condition for measures, and (2) a generalized Poincar\'e inequality in terms of upper gradients.

Cheeger \cite{Cheeger} proved a deep generalization of the Rademacher theorem for the class of metric spaces supporting these two hypotheses.  Though differentiability is a phenomenon enjoyed by Euclidean spaces, the Cheeger and Pansu theorems imply that the geometry of many exotic metric spaces, including Carnot groups and Laakso spaces \cite{Laakso}, is far from Euclidean.  Specifically, such spaces do not allow isometric (or even bi-Lipschitz) embeddings into any $\R^n$, for any $n \in \N$.  More recently, Cheeger and Kleiner \cite{Cheeger:Kleiner}, \cite{Cheeger:Kleiner:gafa}, \cite{Cheeger:Kleiner:annals} have extended these non-embeddability theorems to the case of Lipschitz maps taking values in Banach spaces that satisfy the Radon-Nikod\'ym property.

\subsection{Differentiability on Metric Spaces} \label{sect_intro}

We begin with the spaces of interest.  The discussion below follows the formulation by Keith \cite{Keith}, who gave a further generalization of Cheeger's theorem.

\begin{defn} \label{defn_doubling}
Let $(X,d)$ be a metric space.  A Borel measure $\mu$ on $X$ is called ($\kappa$-){\em doubling} if there exists a constant $\kappa \geq 1$ so that 
$$
0 \;<\; \mu(B(x,2r)) \;\leq\; \kappa\, \mu(B(x,r)) \;<\; \infty
$$
holds for all $x \in X$ and $r > 0$.  We call $Q := \log_2(\kappa)$ the {\em doubling exponent} of $X$.
\end{defn}

As examples, Lebesgue measure on $\R^n$ is doubling; so is the volume element of a compact Riemannian manifold.  In contrast, there also exist doubling measures on $\R^n$ that are singular to Lebesgue measure; for examples, see \cite{Kaufman:Wu} and \cite{Wu}.

To obtain a reasonable theory of calculus, we will need analogues for the gradient of a function.  Following \cite{Semmes} and \cite{Cheeger}, it suffices to work with generalizations for the norm of the gradient.

\begin{defn} \label{defn_ptlip}
On a metric space $(X,d)$, the (upper) 
pointwise Lipschitz constants of a function $f \in \Lip(X)$ 
are defined, respectively, as
\begin{eqnarray*}
{\rm lip}[f](x) &:=&
\liminf_{r \to 0} \Big(
\sup_{y \in \bar{B}(x,r)} \frac{|f(y) - f(x)|}{r}
\Big), \\
\Lip[f](x) 
&:=&
\limsup_{r \to 0} \Big(
\sup_{y \in \bar{B}(x,r)} \frac{|f(y) - f(x)|}{r} 
\Big) \,=\, 
\limsup_{y \to x} \frac{|f(y) - f(x)|}{d(x,y)}.
\end{eqnarray*}
\end{defn}

Pointwise Lipschitz constants are special cases of (weak) upper gradients, for which a robust theory of Sobolev spaces has been developed.  For more details, see \cite{Heinonen:Koskela}, \cite{Shan}, \cite{Hajlasz2}, and \cite{HeinonenLA}.

We now extend the notion of differentiable structure from manifolds to metric measure spaces.  Roughly speaking, it ensures that the Rademacher theorem holds for such spaces.

\begin{defn} \label{defn_measdiff}
Let $(X,d,\mu)$ be a metric measure space.

(1) 
A measurable subset $Y \subset X$ is a {\em chart (of differentiability) on $X$}, if $\mu(Y) > 0$ and there exist $n \in \N$ and a Lipschitz map $\xi: X \to \R^n$ with the following property:\
for every $f \in \Lip(X)$ there is a unique $Df \in L^\infty(Y;\R^n)$ so that
for $\mu$-a.e.\ $x \in Y$,
\begin{equation} \label{eq_diffprop}
\left.\begin{split}
\hspace{.1in}
0 &\;=\;
\Lip\big[ f - Df(x) \cdot \xi \big](x) \\ &\;=\;
\lim_{r \to 0} 
\sup \left\{
\frac{\big|f(y) - f(x) - Df(x)\cdot[\xi(y)-\xi(x)]|}{r} \,:\,
y \in \bar{B}(x,r)
\right\}.
\hspace{.1in}
\end{split}\right\}
\end{equation}
Suggestively, we call $\xi$ a set of {\em coordinates} on $Y$, $Df(x)$ the {\em (measurable) differential of $f$ at $x$ (with respect to $\xi$)}, and $n$ the {\em chart dimension} of $Y$.

\vspace{.05in}
(2) A space $(X,d,\mu)$ supports a {\em (strong) measurable differentiable structure}, if 
there exist $\mu$-measurable subsets $\{X_m\}_{m=1}^\infty$ of $X$, called an {\em atlas} of $X$, so that
\begin{itemize}
\item the set $X \setminus \bigcup_{m=1}^\infty X_m$ has zero $\mu$-measure;
\item each $X_m$ is a chart of differentiability on $X$;
\item there exists $N \in \N$ so that the dimension $n(m)$ of every $X_m$ satisfies $$
0 \leq n(m) \leq N.
$$
\end{itemize}
Such a structure is called {\em non-degenerate} if $n(m) \geq 1$ holds for some $m \in \N$.
\end{defn}



The Cheeger and Keith differentiability theorems are stated below.  Though Poincar\'e inequalities will not be discussed here, we remind the reader that on metric spaces equipped with doubling measures, the validity of a Poincar\'e inequality (in terms of upper gradients) implies the Lip-lip condition \cite[Prop 4.3.1]{Keith}.

\begin{thm}[Cheeger, 1999] \label{thm_cheeger}
Let $(X,d)$ be a metric space and let $\mu$ be a $\kappa$-doubling measure on $X$.  If $(X,d,\mu)$ supports a $p$-Poincar\'e inequality for some $p \geq 1$, then it admits a non-degenerate measurable differentiable structure.  
\end{thm}

\begin{thm}[Keith, 2004] \label{thm_keith}
Let $(X,d)$ be a metric space and let $\mu$ be a doubling measure on $X$.  If $(X,d,\mu)$ satisfies, for some $K \geq 1$, the Lip-lip condition
\begin{equation} \label{eq_Liplip}
\Lip[f](x) \;\leq\; K \, {\rm lip}[f](x)
\end{equation}
for all Lipschitz functions $f: X \to \R$ and for $\mu$-a.e.\ $x \in X$, then it admits a measurable differentiable structure.  
\end{thm}


\subsection{New Results}

As indicated before, in this paper we discuss a new differentiability theorem of Rademacher type.

In particular, our main result {\em characterises} measurable differentiable structures on metric spaces that support doubling measures.  It is also a partial converse to the Cheeger and Keith theorems, in that consequences of their results provide hypotheses for ours.  A brief discussion of these hypotheses is therefore in order.

\subsubsection{Derivations}
One hypothesis is the existence of linear operators on bounded Lipschitz functions called (metric) derivations, as introduced by Weaver \cite{WeaverED}: 
$$
\d : \Lip_b(X) \to L^\infty(X;\mu)
$$
Briefly, these are generalizations of differential operators to the setting of metric measure spaces, with similar algebraic and continuity properties; see Definition \ref{defn_deriv}.  Since the zero map satisfies these conditions, the goal is to study spaces with {\em nontrivial} derivations.

Like vector fields on a Riemannian manifold, derivations on a fixed space have a linear algebraic structure, so the usual notions of linear independence, basis, and pushforward apply to them.

\subsubsection{Lip-derivation inequalities}
Suppose that a non-degenerate measurable differentiable structure exists on a given space $X$.  Indeed, if Equation \eqref{eq_diffprop} holds on a chart $X_m$ of $X$, then every Lipschitz function $f: X \to \R$ satisfies
$$
\Lip[f](x_0) \;=\;
\Lip\big[ D_mf(x_0) \cdot \xi_m \big](x_0) \;\leq\;
\sqrt{n(m)} L(\xi_m) \, \big|D_mf(x_0)\big|
$$
for $\mu$-a.e.\ $x_0 \in X_m$, and where the notation $D_mf = Df$ indicates the dependence on charts.  As observed by Cheeger \cite[Lemma 4.32]{Cheeger} the opposite inequality also holds, 
once a finer atlas is chosen for the space: see also Lemma \ref{lemma_lipdiff}.

If the differential $D_mf = (D_m^1f, \cdots, D_m^nf)$ is replaced by a linearly independent set of derivations ${\bf d} := (\d_i)_{i=1}^n$ on $X_m$ acting on $f$, then we call such a (two-sided) inequality a {\em Lip-derivation inequality}.  More precisely, there exists $K \geq 1$ so that
\begin{equation} \label{eq_lipderiv}
K^{-1} \, |{\bf d}f(x)| \;\leq\;
\Lip[f](x) \;\leq\; 
K \, |{\bf d}f(x)|
\end{equation}
for all $f \in \Lip(X)$ and for $\mu$-a.e.\ $x \in X_m$.

\vspace{.05in}

With these hypotheses, we now present our main result.

\begin{thm} \label{thm_doublingdiff}
Let $(X,d)$ be a metric space, let $\mu$ be a doubling measure on $X$, and let $\{X_m\}_{m=1}^\infty$ be a collection of $\mu$-measurable subsets of $X$, with
$$
\mu\Big( X \setminus \bigcup_{m=1}^\infty X_m \Big) \;=\; 0
$$
and $\mu(X_m) > 0$, for all $m \in \N$.
Then the following conditions are equivalent:
\begin{enumerate}
\item $(X,d,\mu)$ supports a non-degenerate measurable differentiable structure, with charts $\{(X_m,\xi^m)\}_{m=1}^\infty$;
\vspace{.025in}
\item On each $X_m$, there is a linearly independent set of derivations on $X_m$ 
so that 
inequality \eqref{eq_lipderiv} holds for all Lipschitz functions on $X$.
\end{enumerate}
\end{thm}

\begin{rmk}
The direction ``(2) $\Rightarrow$ (1)'' extends Pansu's theorem from Carnot groups to metric spaces that support doubling measures.  Indeed, the associated horizontal vector fields on a Carnot group are well-defined derivations \cite[Thm 39]{WeaverED}.  The novelty here is that the {\em bracket-generating condition}, which ensures a well-defined metric from these vector fields, can be substantially weakened to the Lip-derivation inequality.

On the other hand, Condition (2) is a {\em linear} hypothesis on the space.  Both Theorems \ref{thm_cheeger} and \ref{thm_keith} have non-linear hypotheses but follow from {\em non-constructive} proofs, in that the differential map $f \mapsto D_mf$ arises from abstract ``dimensional'' arguments for generalized linear functions.  It would be of interest if one could prove a Rademacher-type theorem where the differential map is explicitly constructed, such as in the analysis on fractals \cite{Kigami}.

Recently Schioppa \cite[Thm 5.9]{Schioppa} has generalised the direction (2) $\Rightarrow$ (1), where one only requires a {\em one-sided} Lip-derivation inequality and where the constant $K_m$ can depend on the point.  Moreover, he shows that the linearly independent sets of derivations in Theorem \ref{thm_doublingdiff} are in fact bases \cite[Cor 6.15]{Schioppa}. 
\end{rmk}

\begin{rmk}
The other direction ``(1) $\Rightarrow$ (2)'' gives a new proof that spaces supporting doubling measures and Lip-lip conditions also support nontrivial derivations.  The case of spaces $X$ supporting Poincar\'e inequalities was shown earlier by Cheeger and Weaver \cite[Thm 43]{WeaverED}.  Our proof, like theirs, relies on a robust theory of Sobolev functions on such spaces. 
\end{rmk}

In the latter direction, Theorem \ref{thm_doublingdiff} requires the crucial property (Lemma \ref{lemma_reflexive}) of reflexivity for the Haj{\l}asz-Sobolev spaces $M^{1,p}(X)$, for $p > 1$.  It is worth noting that for certain fractal subsets $S$ of $\R^n$ equipped with their natural self-similar measures, $M^{1,p}(S)$ is neither separable nor reflexive for any $p \in (1,\infty)$ \cite{Rissanen}.  As a consequence, this gives a new non-differentiability result for such fractals.

\begin{cor}
Let $K$ be a self-similar fractal of Cantor type in $\R^n$.  If $\H$ is the natural self-similar (Hausdorff) measure associated to $K$, then $K$ does not support a non-degenerate measurable differentiable structure with respect to $\H$.
\end{cor}

To clarify, such sets $K$ are constructed as invariant subsets under similitude maps $S_j : \R^n \to \R^n$, $j = 1, 2, \ldots, N$, of the form
$S_j(x) = \lambda_j (R_jx) + v_j$,
for fixed $\lambda_j \in (0,1)$, $R_j \in SO(n,\R)$, and $v_j \in \R^n$.  The invariance then reads as
$$
K \;=\; \bigcup_{i=1}^N S_j(K).
$$
Moreover, $K$ is {\em of Cantor type} if $S_i(K) \cap S_j(K) = \emptyset$ holds whenever $i \neq j$.

We note that there exist self-similar fractals, not of Cantor type, but still lack such structures.  For example, the middle-thirds Sierpi\'nski carpet admits a degenerate measurable differentiable structure with respect to its natural Hausdorff measure; in fact, it supports no nonzero derivations \cite[Thm 41]{WeaverED}.  It would be interesting to determine a sharp criterion for fractals with measurable differentiable structures, but to the author's knowledge, such results remain unknown.

\subsection{Regarding the doubling condition}

In some sense, the doubling condition in Theorem \ref{thm_doublingdiff} is close to necessary.  Indeed, Bate and Speight \cite{Bate:Speight} proved that if a space $(X,d,\mu)$ supports a measurable differentiable structure, then $\mu$ must be pointwise doubling, but not necessarily with a uniform constant $\kappa$; that is, for $\mu$-a.e.\ $x \in X$ we have
$$
0 \;<\;
\limsup_{r \to 0} \frac{\mu(B(x,2r))}{\mu(B(x,r))} \;<\;
\infty.
$$

Returning to the setting of (uniform) doubling measures, the key step in the proof of Theorem \ref{thm_doublingdiff} is a new fact of possibly independent interest.

\begin{lemma} \label{lemma_doublingbd} 
On metric spaces supporting $\kappa$-doubling measures, the module of derivations is necessarily of finite rank, and the rank bound depends only on $\kappa$.
\end{lemma}

The above lemma ensures that there is a uniform rank bound to the basis of each $\U(X_m,\mu)$.  Put in effect, 
this gives a fixed dimension for the measurable differentiable structure in Theorem \ref{thm_doublingdiff}.

As for Lemma \ref{lemma_doublingbd}, its proof requires ``snowflaking'' the given space and applying a variant of Assouad's embedding theorem \cite{Assouad}.  In some sense the result is surprising, since snowflaked metric spaces do not support nontrivial derivations in general \cite[Thm 36]{WeaverED}.  To avoid this apparent impasse, one takes Lipschitz approximations of the embedding and its inverse separately.  Subsequently, pushforward derivations on Euclidean spaces can then be used without assuming any injectivity of the Lipschitz maps.

In a similar direction, Lang and Z\"ust \cite{Lang:Zuest} have proved a version of Lemma \ref{lemma_doublingbd} for currents on metric spaces.  Though it is known \cite{Gong_currents} that $k$-dimensional currents induce bases of derivations of rank-$k$, the result of Lang and Z\"ust applies to a larger class of spaces --- namely, those with finite {\em Nagata dimension} which, as studied by Lang and Schlichenmaier, includes the case of doubling measures \cite{Lang:Schlichenmaier}.  Related to this, Z\"ust \cite{Zust} has also used Assouad's embedding to show that {\em normal currents} on doubling metric spaces are exactly pushforwards of Euclidean currents.

\subsection{Connections to the Lip-lip condition}
The method of using derivations to prove differentiability theorems applies to other settings as well. As one example, the Rademacher property holds for metric measure spaces that satisfy the Lip-derivation inequality and on which bounded Lipschitz functions form a finitely generated algebra (Theorem \ref{thm_fingen-measdiff}).

It is not known if the Lip-lip condition is necessary for Rademacher-type theorems on metric measure spaces.  This motivates the following open question.

\begin{ques} \label{ques_indep}
Is there a metric space that supports a doubling measure and a nondegenerate measurable differentiable structure, but where the Lip-Lip condition fails on a set of positive measure?
\end{ques}

A theorem of Cheeger \cite[Thm 14.2]{Cheeger} states that for spaces which support doubling measures and Poincar\'e inequalities, as well as an additional measure density condition \cite[Conj 4.63]{Cheeger}, the images of charts $X_m$ under coordinates $\xi_m$ must be $n(m)$-rectifiable sets.  It is reasonable to expect, more generally, that doubling spaces satisfying the Lip-lip condition would also enjoy the same geometric rigidity, but to the author's knowledge, such a result also remains unknown.

\vspace{.05in}

{\bf Plan of the Paper}.\ Section \S1 has provided an introduction to the work and a summary of our main results.  Section \S\ref{sect_prelim} reviews basic facts about Lipschitz functions and derivations on metric measure spaces.

To motivate the proof ideas later, Section \S\ref{sect_fingen} begins with metric spaces on which bounded Lipschitz functions form a finitely-generated algebra; the existence of measurable differentials, in such settings, becomes a Euclidean matter.

The case of doubling measures is treated in Section \S\ref{sect_doubling}, which includes the fact that the doubling condition imposes a rank bound for derivations.  In Section \S\ref{sect_suff} we address the necessity of nontrivial derivations and Lip-derivation inequalities for measurable differentiable structures.

\vspace{.1in}

{\bf Acknowledgments}.\ The author thanks John Mackay, Pekka Pankka, Elefterios Soultanis, Jeremy Tyson, Antti V\"ah\"akangas, Kevin Wildrick, Xiao Zhong, and Thomas Z\"urcher for helpful discussions which led to improvements in this work.  He also thanks the Oberwolfach Mathematics Institute for the hospitality during the Seminar in Lipschitz Analysis, held in November 2010, which inspired many ideas in this work.

\section{Preliminaries} \label{sect_prelim}

Here and in the sequel we will consider only metric measure spaces $(X,d,\mu)$, that is: metric spaces $(X,d)$ equipped with Borel measures $\mu$.  Moreover, the metric spaces in question are always assumed to be separable.  Several classes of functions will often appear in the paper:

$P_n$, the set of all polynomials in $n$ variables, with coefficients in $\R$,

$\Lip(X)$, the set of all Lipschitz functions on $X$,

$\Lip_b(X)$, the set of all bounded Lipschitz functions on $X$.

\subsection{Lipschitz functions}

The Lipschitz constant of a function $f: X \to \R$ is
$$
L(f) \;=\; \sup\left\{ \frac{|f(y)-f(x)|}{d(x,y)} \,;\, x,y \in X,\, x \neq y\right\}
$$
and if $L(f) \leq K$, then $f$ is called $K$-Lipschitz.

The proofs in later sections also use pointwise Lipschitz constants, defined in \S\ref{sect_intro}, as a replacement for the norm of the gradient.  We begin with a weak version of the Chain Rule for pointwise Lipschitz constants.

\begin{lemma} \label{lemma_polylipzero}
Let $f = (f_i)_{i=1}^n \in [\Lip(X)]^n$ and $x \in X$.  If $\Lip[f_i](x) = 0$ holds for each $1 \leq i \leq n$, then $\Lip[p \circ f](x) = 0$ for all $p \in P_n$.
\end{lemma}

\begin{proof}
It suffices to check monomials 
$
p(y) = \prod_{i=1}^n y_i^{m_i} 
$,
for $\{m_i\}_{i=1}^n$ in $\N$.  We proceed by induction, so for $n=1$, put $f = f_1$.  Since $\Lip[f](x) = 0$, it follows that $f$ is continuous at $x$.  
For $m = m_1 > 1$, we estimate 
\begin{eqnarray*}
\Lip[f^m](x) &=&
\limsup_{y \to x} \frac{|f(y)^m-f(x)^m|}{d(x,y)} \\ &\leq&
\limsup_{y \to x} \Big\{ 
\frac{|f(y)-f(x)|}{d(x,y)} \sum_{a=1}^m |f(y)|^{m-a}|f(x)|^{a-1} \Big\} \\ &=&
m|f(x)|^{m-1} \cdot \Lip[f](x) \;=\; 0.
\end{eqnarray*}
As for $n \geq 1$, we use the Triangle inequality, include auxiliary terms
$$
f_1(x)^{m_1} \prod_{i=2}^n f_i(y)^{m_i} \text{ and }
f_1(x)^{m_1}f_2(x)^{m_2} \prod_{i=3}^n f_i(y)^{m_i} \text{ and so on, }
$$
and estimate similarly as before.
\end{proof}

We proceed with two more facts about Lipschitz functions.  For their proofs, see \cite{McShane} and \cite{Arens:Eells}, respectively.

\begin{lemma}[McShane,Whitney] \label{lemma_mcshane}
For a metric space $(X,d)$ and for $A \subset X$, each $f \in \Lip(A)$ admits a $L(f)$-Lipschitz extension to all of $X$, given by
$$
x \,\mapsto\, \inf \big\{ f(a) \,+\, L(f) \cdot d(x,a) \,:\, a \in A \big\}.
$$
\end{lemma}

\begin{lemma}[Arens-Eells] \label{lemma_dualbanach}
If $X$ is a metric space, then 
$\Lip_b(X)$ is a dual Banach space with respect to the norm
$$
\|f\|_{\Lip} \;:=\; \max\{ L(f), \|f\|_\infty \}.
$$
Moreover, on bounded subsets of $\Lip_b(X)$, the topology of weak-$*$ convergence agrees with that of pointwise convergence.
\end{lemma}

\begin{rmk} \label{rmk_net}
Since metric spaces $X$ are assumed separable, the weak-$*$ topology in $\Lip_b(X)$ can be characterized in terms of sequences as opposed to nets.  The {\em Arens-Eells space}, a pre-dual of $\Lip_b(X)$, is therefore a separable Banach space whenever $X$ is separable \cite[Sect 2.2]{WeaverLA}.

So in this context, a sequence $\{f_m\}$ converges weak-$*$ to $f$ in $\Lip_b(X)$, denoted $f_m \wsto f$, if and only if both $\{ f_m \}$ converges pointwise to $f$ and $\sup_m L(f_m) < \infty$.
\end{rmk}

Recall that $\Lip_b(X)$ is an algebra: if $f, g \in \Lip_b(X)$, then $f \cdot g \in \Lip_b(X)$ and
$$
L(f \cdot g) \;\leq\; L(g) \cdot \|f\|_\infty \,+\, L(f) \cdot \|g\|_\infty \;<\; \infty.
$$

\begin{defn} \label{defn_fingen}
For $N \in \N \cup \{\infty\}$, a subset ${\bf g} = (g_i)_{i=1}^N$ of $\Lip(X)$ is said to {\em generate} $\Lip_b(X)$ if on every ball $B$ in $X$ and
for every $f \in \Lip_b(X)$, there exist polynomials $\{p_n\}_{n=1}^\infty \in P_n \cap \Lip_b(g(B))$ so that 
$$
p_n \circ g \;\wsto\; f \, \text{ in } \, \Lip_b(X).
$$
Call $\Lip_b(X)$ {\em $N$-generated}, for $N \in \N$, if there exists an $N$-tuple in $[\Lip(X)]^N$ that generates $\Lip_b(X)$ and if no $(N-1)$-tuple generates $\Lip_b(X)$.  Lastly, $\Lip_b(X)$ is {\em finitely generated} if it is $N$-generated for some $N \in \N$.
\end{defn}

As an example, $\R^n$ is $n$-generated.  Indeed, it is well-known that polynomials are dense in $C^\infty(\R^n)$ with respect to the $C^1$-topology and that smooth functions are norm-dense in $\Lip_b(\R^n)$.  

\subsection{Derivations and Locality}

This discussion is adapted from \cite{WeaverED}, which handles the general case of measurable metrics. For (pointwise) metrics in the usual sense, see \cite{HeinonenNC}, \cite{Gong}, and the recent work \cite{Schioppa}.

\begin{defn} \label{defn_deriv}
A bounded linear operator $\d : \Lip_b(X) \to L^\infty(X;\mu)$ is called a {\em derivation} if it satisfies the following two conditions:
\begin{enumerate}
\item the Leibniz rule: $\d(f \cdot g) = f \cdot \d g \,+\, g \cdot \d f$.
\item weak continuity: if $f_k \wsto f$ in $\Lip_b(X)$, then $\d f_k \wsto \d f$ in $L^\infty(X;\mu)$.
\end{enumerate}
The set of derivations on $(X,d,\mu)$ is denoted by $\U(X,\mu)$.
\end{defn}

As examples, the differential operators $\{\frac{\partial}{\partial x_i} \}_{i=1}^n$ are derivations on $\R^n$ with respect to the usual metric and the Lebesgue measure; so are vector fields on a compact Riemannian manifold with respect to the volume element \cite[Thm 37]{WeaverED}.  On the other hand, measures that are supported on finite sets of points do not support nonzero derivations \cite[Prop 32]{WeaverED}.

Observe that derivations allow scaling by $L^\infty$ functions; if $\d \in \U(X,\mu)$ then each $\lambda \in L^\infty(X;\mu)$ determines a derivation $\lambda \d \in \U(X,\mu)$ under the rule
$$
(\lambda \d)f(x) \;:=\; \lambda(x) \, \d f(x).
$$
Returning to the analogy of differential operators on $\R^n$, derivations therefore enjoy a locality property \cite[Lem 27]{WeaverED}. As a consequence, they also allow a well-defined action on unbounded Lipschitz functions \cite[Thm 2.15]{Gong_rigid}.

\begin{lemma}[Weaver] \label{lemma_locality}
Let $A \subset X$ with $\mu(A) > 0$.  Then as sets,
$$
\U(A,\mu) \;=\; \{ \chi_A\d \,:\, \d \in \U(X,\mu) \}.
$$
\end{lemma}

\begin{lemma} \label{lemma_extlip}
Each $\d \in \U(X,\mu)$ extends to a linear operator
$$
\bar\d : \Lip_{\rm loc}(X) \to L^\infty_{\rm loc}(X,\mu).
$$
Moreover, it is bounded on each compact subset $K$ of $X$ under the seminorm
$$
f \mapsto L(f|_K).
$$
\end{lemma}

In light of the above discussion, we henceforth make no distinction between a derivation (as in Definition \ref{defn_deriv}) and its extension to $\Lip_{\rm loc}(X)$ (as in Lemma \ref{lemma_extlip}).

\subsection{Linear independence \& Rank}

We now consider more subtle consequences of the scalar action $L^\infty(X;\mu)$ on $\U(X,\mu)$.

\begin{defn}
A set $\{\d_i\}_{i=1}^m$ in $\U(X,\mu)$ is called {\em linearly independent} if every $m$-tuple of functions $\{\lambda_i\}_{i=1}^m$ in $L^\infty(X;\mu)$ satisfies the implication
$$
\big[ \lambda_1 \d_1 \,+\, \ldots \,+\, \lambda_m\d_m \,=\, 0 \big]
\; \; \Longrightarrow \; \;
\big[ \lambda_1 \,=\, \ldots \,=\, \lambda_m \,=\, 0 \big].
$$
Otherwise, call $\{\d_i\}_{i=1}^m$ a {\em linearly dependent} set.

Moreover, {\em $\U(X,\mu)$ has rank-$m$} if it contains a linearly independent set of $m$ derivations and if every set of $m+1$ derivations is linearly dependent.  Lastly, call $\{\d_i\}_{i=1}^m$ a {\em basis} of $\U(X,\mu)$ if it is linearly independent and if $\U(X,\mu)$ has rank-$m$.
\end{defn}

The linear algebra of derivations will be used extensively in later sections.  The basic idea is to use generating functions for $\Lip_b(X)$ as coordinates for $X$.  By forming a Jacobi-type matrix whose entries consist of derivations acting on these functions, we construct differentials using a ``change of variables'' argument.

We begin with a few lemmas.  The first three generalise the orthogonal relations 
$$
\frac{\partial}{\partial x_i}[x_j] \;=\; 
\left\{
\begin{array}{rl}
0,& i \neq j \\
1,& i = j
\end{array}
\right.
$$
where $\{x_j\}$ are the usual coordinate functions on $\R^n$.

\begin{lemma} \label{lemma_nonsing}
Let $n \in \N$.  If ${\bf g} := (g_i)_{i=1}^n$ is a generating set for $\Lip_b(X)$ and if ${\bf d} := (\d_i)_{i=1}^n$ is a linearly independent set in $\U(X,\mu)$, then 
$$
\mathbf{d g}(x) \;:=\; [\d_ig_j(x)]_{i,j=1}^n
$$
is a non-singular matrix for $\mu$-a.e. $x \in X$.
\end{lemma}

Suggestively, ${\bf dg}$ is called the {\em Jacobi matrix of ${\bf g}$ (with respect to ${\bf d}$)} and its determinant $\det({\bf dg})$ is called the {\em Jacobian (determinant)} of ${\bf g}$.  As a clarification, we follow the usual Jacobi matrix convention on Euclidean spaces, so $i$ is the column index and $j$ is the row index.

\begin{proof}[Proof of Lemma \ref{lemma_nonsing}]
Since ${\bf d}$ is linearly independent and ${\bf g}$ generates $\Lip_b(X)$, not all of the entries of ${\bf dg}(x)$ can be zero. Towards a contradiction, let $k \in (1,n]$ be the least integer with the following properties:
\begin{enumerate}
\item there is a $k \times k$ cofactor matrix $A(x)$, obtained from omitting $n-k$ rows and $n-k$ columns from ${\bf dg}(x)$, so that the set
$$
Y \,:=\, \big\{ x \in X \,:\, \det A(x) \,=\, 0 \big\}
$$
has positive $\mu$-measure;
\vspace{.05in}
\item there is a $(k-1) \times (k-1)$ cofactor matrix $A'$, obtained from omitting one row and one column from $A$, so that $\det(A')|_Y \neq 0$.
\end{enumerate}
In particular, $1 \times 1$ cofactors are precisely the entries $\d_ig_j$, so necessarily $k \geq 2$.

Up to re-indexing, let $A := [\d_ig_j(x)]_{i,j=1}^k$.  Writing $A_j$ for the cofactor of $A$ with the first row and $j$th column of $A$ omitted, suppose that 
$A' := A_k$.  Then
\begin{equation} \label{eq_cofactorderiv}
\d \;=\; \sum_{i=1}^k \big( \chi_Y (-1)^{j+1} \det A_j \big) \d_i
\end{equation}
is zero; verily, the Laplace expansion formula for matrices implies that $\d g_j$ is either $\det(A)$ or the determinant of another $k\times k$ cofactor matrix with a repeated row.

Since $\det A' \neq 0$ on $Y$, the derivations $\{\chi_Y\d_i\}_{i=1}^k$ must be linearly dependent, which implies that $\{\d_i\}_{i=1}^k$ and hence ${\bf d}$ are also linearly dependent.
\end{proof}

The non-singular Jacobian condition also holds for when the number of generators for $\Lip_b(X)$ exceeds the rank of $\U(X,\mu)$.

\begin{cor} \label{cor_nonsing}
Let $m,n \in \N$ with $m \leq n$.  If ${\bf g} = (g_i)_{i=1}^n$ generates $\Lip_b(X)$ and if ${\bf d} := \{\d_i\}_{i=1}^m$ is linearly independent in $\U(X,\mu)$, then for $\mu$-a.e.\ $x \in X$ there is a subset ${\bf f} := (f_j)_{j=1}^m$ of ${\bf g}$ so that 
${\bf df}(x) := [\d_if_j(x)]_{i,j=1}^m$
is a non-singular matrix.
\end{cor}

The proof is a straightforward induction on $m$; for the induction step, one argues in the contrapositive by using Equation \eqref{eq_cofactorderiv}.  In fact, the same argument works even when $\Lip_b(X)$ is generated by {\em countably} many Lipschitz functions.

As another consequence, we obtain a type of Gram-Schmidt orthogonalization for bases of derivations.

\begin{lemma}  \label{lemma_ON} 
Let $n \in \N \cup \{\infty\}$.  If ${\bf g} := (g_i)_{i=1}^n$ is a generating set for $\Lip_b(X)$ and if $\U(X,\mu)$ has rank-$m$, for some finite $m \in (0, n]$, then there exist
\begin{itemize}
\item a basis $\{\d_i^*\}_{i=1}^m$ of $\U(X,\mu)$,
\item a partition of $X$ by $\mu$-measurable sets $\{X_l\}_{l=1}^L$, with $L \leq \binom{n}{m}$  when $n < \infty$,
\end{itemize}
where for every $1 \leq l \leq L$, there is a subset ${\bf f}^l := (f^l_j)_{j=1}^m$ of ${\bf g}$ so that:
\begin{enumerate}
\item if $i \neq j$, then $\d_i^*f_j^l \;=\; 0$ holds $\mu$-a.e.\ on $X_l$;
\item the set $\{ x \in X_l \,:\, \d_1^*f_1^l(x) = 0\}$ has zero $\mu$-measure;
\item $\d_1^*f_1^l = \ldots = \d_n^*f_n^l \neq 0$ holds $\mu$-a.e.\ on $X_l$.
\end{enumerate}
\end{lemma}

Such a basis of $\U(X,\mu)$ is called {\em orthogonal} (with respect to ${\bf g}$).

\begin{proof}
We form the partition first, and then construct the basis for $\U(X,\mu)$.

\vspace{.05in}
{\em Step 1:\ Partitioning}.\
Let ${\bf d} := \{\d_i\}_{i=1}^m$ be a basis of $\U(X,\mu)$.  For each subset ${\bf f} = (f_j)_{j=1}^m$ of ${\bf g}$ with cardinality  $m = {\rm rank}[\U(X,\mu)]$, define
$$
X_{\bf f} \;:=\; \{ x \in X \,:\, \det[{\bf df}(x)] \neq 0 \}.
$$
By Corollary \ref{cor_nonsing}, at least one of the sets $X_{\bf f}$ has positive $\mu$-measure and
$$
\mu\big(X \setminus \Big( \bigcup_{\bf f} X_{\bf f}\Big) \big) \;=\; 0.
$$
So up to omitting duplicates $X_{\bf f} = X_{\bf f'}$ and re-indexing, the partition consists of the collection $\{X_l\}_{l=1}^L = \{X_{\bf f}\}$, with cardinality $L \leq \binom{n}{m}$ whenever $n < \infty$.

The collection $\{X_{\bf f}\}$ can also be assumed to be pairwise disjoint, by taking intersections and (relative) complements of sets as necessary.

\vspace{.05in}

{\em Step 2:\ Bases of derivations}.\ For each $1 \leq i \leq m$, let $\d_i^{\bf f} = \d$ be the deriviation defined as in Equation \eqref{eq_cofactorderiv}, with the functions $f_j$ in place of the $g_j$ and with the $(m-1) \times (m-1)$ cofactor matrix of ${\bf df}(x)$, obtained by omitting the first row and $j$th column, in place of the $A_j$.

Indeed the set $\{\d_i^{\bf f}\}_{i=1}^n$ satisfies conclusions (1) and (3) $\mu$-a.e.\ on $X$, purely by properties of determinants, and conclusion (2) for $\{\d_i^{\bf f}\}_{i=1}^n$ follows from Corollary \ref{cor_nonsing}, with $X_{\bf f}$ in place of $X$.  By inspection, the derivations
$$
\d_i^* \;:=\; \sum_{\bf f} \chi_{X_{\bf f}} \, \d_i^{\bf f}
$$
also satisfy the same conclusions, with $X$ in place of $X_l$ for (2).

It remains to show that $\{\d_i^*\}_{i=1}^n$ is linearly independent, so it suffices to check $\{\d_i^{\bf f}\}_{i=1}^m$ on each $X_{\bf f}$.  Suppose there are functions $\{\lambda_i\}_{i=1}^n$ in $L^\infty(X,\mu)$ so that $\sum_{i=1}^N \lambda_i \d_i^{\bf f} = 0$.  In particular, for each generator $g_j$, conclusion (1) implies that
$$
0 \;=\;
\Big(\sum_{i=1}^n \lambda_i \d_i^{\bf f} \Big)g_j \;=\;
\lambda_j \, (\d_j^{\bf f} g_j).
$$
By conclusion (2), $\d_j^{\bf f} g_j$ is nonzero $\mu$-a.e.\ on $X_{\bf f}$, so $\lambda_j = 0$ holds $\mu$-a.e.\ on $X_{\bf f}$.
\end{proof}

Using the above linear algebraic properties, we now show how rank bounds for derivations follow from the finitely-generated property of a Lipschitz algebra.

\begin{lemma} \label{lemma_fingenbd}
If $\Lip_b(X)$ is $n$-generated, for $n \in \N$, then ${\rm rank}[\U(X,\mu)] \leq n$.
\end{lemma}

\begin{proof}
Let $\{g_i\}_{i=1}^n$ generate $\Lip_b(X)$, and suppose instead that $\{\d_j\}_{j=1}^m$ is a basis for $\U(X,\mu)$, for some $m > n$.  Since $\d_{n+1} \neq 0$, let ${\bf g} := (g_{k_j})$ be the tuple of generators for which $\lambda_j := \d_{n+1}g_{k_j}$ are not identically zero in $L^\infty(X,\mu)$.

Since ${\bf d} := \{\d_i\}_{i=1}^n$ is also linearly independent, let $\{\d^*_j\}_{j=1}^n$ be the corresponding derivations from Lemma \ref{lemma_ON}, and put $\d_{n+1}^* := \d_{n+1}$.  In particular,
$$
\lambda_{n+1} \;:=\; -\d^*_1g_1 \;\neq\; 0.
$$
Observe that $\sum_{i=1}^{n+1} \lambda_i \d^*_{k_i}$ acts as zero on each $g_j$: indeed, if $g_{k_j} \in {\bf g}$, then 
\begin{eqnarray*}
\d g_{k_j} \;=\;
\sum_{i=1}^{n+1} \lambda_i \, \d^*_{k_i}g_{k_j} &=&
-(\d^*_1g_1) \, \d_{n+1}g_{k_j} \,+\, \sum_{i=1}^n (\d_{n+1}g_{k_i}) \, \d^*_{k_i}g_{k_j} \\ &=&
-(\d^*_{k_j}g_{k_j}) \, \d_{n+1}g_{k_j} \,+\, (\d_{n+1}g_{k_j}) \, \d^*_{k_j}g_{k_j} \;=\; 0.
\end{eqnarray*}
Otherwise $g_j \notin {\bf g}$, so $\d_{n+1}g_j = \lambda_j = 0$ holds and hence $\d g_j = 0$.  Using the Leibniz rule, the same holds for $\d (p \circ g)$, for every $p \in P_n$.

The finitely generated property of $\Lip_b(X)$ and weak continuity for derivations imply that $\d = 0$, which contradicts the linear independence of ${\bf d}$.
\end{proof}

\subsection{Derivations on Euclidean Spaces}

We conclude this section with a few facts that are specific to $\R^n$.  The first is a simple consequence of Lemma \ref{lemma_ON}.

\begin{cor} \label{cor_planezero}
Let $\mu$ be a Radon measure on $\R^n$.  If $\U(\R^n,\mu)$ has rank $n$, then every affine hyperplane has zero $\mu$-measure.
\end{cor}

\begin{proof}
Supposing otherwise, let $\mathcal{P}$ be a hyperplane with $\mu(\mathcal{P}) > 0$.  Choose a linear coordinate system $\{y_j\}_{j=1}^n$ on $\R^n$ so that $\mathcal{P} = \{ y_1 = 0 \}$.  Since 
$y_1|_\mathcal{P}$ extends to a constant function on $\R^n$, locality implies that $\d y_1|_\mathcal{P} = 0$ for all $\d \in \U(\R^n,\mu)$.  

By hypothesis, let $\{\d_i\}_{i=1}^n$ be a basis of $\U(\R^n,\mu)$, so by Lemma \ref{lemma_ON}, assume that conclusions (1) and (3) hold for $g_j := y_j$.  Since $\chi_{\mathcal{P}}\d_1$ is a nontrivial linear combination, the desired contradiction follows.
\end{proof}

The next lemma \cite[Lem 2.20]{Gong_rigid} is a generalized Chain Rule for derivations.

\begin{lemma} \label{lemma_chainrule}
Let $\mu$ be a Radon measure on $\R^n$.
For every $f \in \Lip_b(\R^n)$, there exists a vectorfield $v_f = (v_f^1, \cdots, v_f^n) \in L^\infty(\R^n;\R^n;\mu)$ that satisfies
$$
\d f(z) \;=\; \sum_{i=1}^n v_f^i(z) \, \d x_i(z).
$$
for all $\d \in \U(\R^n,\mu)$ and for $\mu$-a.e.\ $z \in \R^n$.  If $f \in C^1(\R^n)$, then $v_f = \nabla f$.
\end{lemma}

The last two facts require the {\em pushforward} of a derivation.  To begin, recall that for a Borel measure $\mu$ on a space $X$ and a Borel function $\xi : X \to Y$, the pushforward measure $\xi_\#\mu$ on $Y$ is defined for Borel measurable subsets $A \subset Y$ as
$$
\xi_\#\mu(A) \;:=\; \mu\big( \xi^{-1}(A) \big).
$$
It is a fact \cite[Thm 1.19]{Mattila} that every $\varphi \in L^1(Y,\xi_\#\mu)$ satisfies
\begin{equation} \label{eq_pushfwd}
\int_{\xi^{-1}(A)} \varphi \, d(\xi_\#\mu) \;=\; \int_A \varphi \circ \xi \, d\mu
\end{equation}
whenever $A$ is a $\mu$-measurable subset of $X$. 
Of the following two lemmas, the first is \cite[Lem 2.17]{Gong} and the second is a consequence of it.

\begin{lemma} \label{lemma_pushfwd}
Let $X,Y$ be metric spaces, let $\mu$ be a Borel measure on $X$, and let $\xi : X \to Y$ be Lipschitz.
For each $\d \in \U(X,\mu)$, there is a unique derivation $\xi_\#\d \in \U(Y,\xi_\#\mu)$ that satisfies, for all $\varphi \in L^1(Y,\xi_\#\mu)$ and all $\pi \in \Lip(Y)$,
$$
\int_Y \varphi (\xi_\#\d)\pi \, d(\xi_\#\mu) \;=\;
\int_X (\varphi \circ \xi) \d(\pi \circ \xi) \, d\mu.
$$
\end{lemma}

\begin{lemma} \label{lemma_fullrank}
Let $(X,d,\mu)$ be a metric measure space so that ${\bf g} := (g_i)_{i=1}^n$ generates $\Lip_b(X)$, for some $n \in \N$.  If $\U(X,\mu)$ has rank-$n$, then so does $\U(\R^n,\xi_\#\mu)$.
\end{lemma}

\begin{proof}[Proof of Lemma \ref{lemma_fullrank}]
As in Lemma \ref{lemma_ON}, let $\{\d_j^*\}_{j=1}^n$ be an orthogonal basis of $\U(X,\mu)$ with respect to ${\bf g}$. 
Suppose there exist functions $\{\Lambda_j\}_{j=1}^n$ in $L^\infty(\R^n,{\bf g}_\#\mu)$ 
so that the linear combination $\sum_{j=1}^n \Lambda_j ({\bf g}_\#\d_j^*)$ is identically zero in $\U(\R^n,{\bf g}_\#\mu)$.

Applying Lemma \ref{lemma_pushfwd} to $Y = \R^n$ and to each $\pi = x_i$, it follows that
\begin{eqnarray*}
0 &=&
\int_{\R^n} \varphi \Big[
\sum_{j=1}^n \Lambda_j({\bf g}_\#\d_j^*)
\Big]
x_i \, d(g_\#\mu) \\ &=&
\sum_{j=1}^n \int_X (\varphi \circ {\bf g}) (\Lambda_j \circ {\bf g})\d_j^*(x_i \circ {\bf g}) \, d\mu \;=\;
\int_X ((\varphi \Lambda_i) \circ {\bf {\bf g}}) \d_i^*g_i \, d\mu
\end{eqnarray*}
holds for all $\varphi \in L^1(\R^n,{\bf g}_\#\mu)$.  By replacing $\d_j^*$ with the rescaled derivation
$$
(\chi_{\{\d_j^*g_j > 0\}} - \chi_{\{\d_j^*g_j < 0\}})\d_j^*
$$
we may assume that $\d_i^*g_i > 0$ holds $\mu$-a.e.\ on $X$.  By further choosing $\varphi = \chi_{B(x,r)}$, it follows that $\Lambda_i \circ {\bf g} = 0$ holds $\mu$-a.e.\ on every ball $B(x,r)$ and therefore $\Lambda_i = 0$ in $L^\infty(\R^n,{\bf g}_\#\mu)$, for each $i = 1, 2, \ldots n$.
\end{proof}

\section{The Case of Finitely Generated Lipschitz Algebras} \label{sect_fingen}

The differentiability theorems in this section are analogues of the Inverse and Implicit Function Theorems from real analysis, but where derivations and generators replace partial derivatives and local coordinates, respectively.  

\begin{thm} \label{thm_dimmatch}
Let $(X,d,\mu)$ be a metric measure space 
so that $\U(X,\mu)$ has rank $N > 0$.
If $\Lip_b(X)$ is $N$-generated and if $X$ satisfies the Lip-derivation inequality \eqref{eq_lipderiv}, then $X$ supports a non-degenerate measurable differentiable structure.
\end{thm}

Measurable differentiable structures also exist on spaces where the number of Lipschitz generators exceeds the rank.  The measure need not be doubling here, either, but only satisfy the Lebesgue differentiation property
\begin{equation} \label{eq_leb}
\frac{1}{\mu(B(x,r))} \int_{B(x,r)} f \,d\mu \;\longrightarrow\; f(x), \text{ as } r \to 0
\end{equation}
for all non-negative $f \in L^1(X,\mu)$ at $\mu$-a.e.\ $x \in X$.

\begin{thm} \label{thm_fingen-measdiff}
Let $(X,d)$ be a metric space and let $\mu$ be a Radon measure that satisfies the Lebesgue differentiation property \eqref{eq_leb}.  If
\begin{enumerate}
\item there is a basis ${\bf d} := \{\d_j\}_{j=1}^M$ of $\U(X,\mu)$, for some $M > 0$,
\item there is a generating set ${\bf g} := \{g_i\}_{i=1}^N$ for $\Lip_b(X)$, with $M \leq N$,
\item the Lip-derivation inequality \eqref{eq_lipderiv} holds, for some $K \geq 1$,
\end{enumerate}
then $(X,d,\mu)$ supports a non-degenerate measurable differentiable structure.
\end{thm}

The proofs of Theorems \ref{thm_dimmatch} and \ref{thm_fingen-measdiff} use dyadic-cube decompositions of Euclidean spaces, as well as piecewise-linear (PL) extensions of functions.  
To fix notation, for an $N$-simplex $S$ in $\R^N$, its set of vertices (or $0$-skeleton) is denoted by $S^0$.

\begin{rmk} \label{rmk_pl}
For each closed $N$-simplex $S$, every $f \in \Lip(S^0)$ has a unique linear extension $F : S \to \R$ that satisfies $L(F) = L(f)$.

The same holds for triangulations by closed $N$-simplices $\{S_m\}_{m=1}^\infty$ of $\R^N$.  Indeed, by separately taking linear extensions from $S_m^0$ to $S_m$, each $f \in \Lip\big(\bigcup_{m=1}^\infty S_m^0\big)$ has a unique PL-extension $F \in \Lip(\R^N)$ that also satisfies $L(F) = L(f)$.

For the remainder of this section we will work with a fixed triangulation of $\R^N$, for each $n \in \N$.  Starting with dyadic points in $\R^N$ at scale $2^{-n}$,
$$
V_1 \,:=\, 
\left\{ 2^{-n}k : k \in \mathbb{Z} \right\} \, \text{ and } \, 
V_N \,:=\, 
\big\{ 
x = (x_1,\ldots,x_N) \in \R^N \,:\, 
x_i \in V_1
\big\},
$$
we fix a subdivision on the cube $[0,2^{-n}]^N$ into finitely many closed $N$-simplices $\{S_m\}$, whose union covers the cube and so that every intersection $S_m \cap S_n$ is either the empty set or a lower-dimensional simplex.

Taking translates in the coordinate directions, this determines the desired triangulation of $\R^N$, so as above, every $f \in \Lip(V_N)$ extends to a function on all of $\R^N$ with the same Lipschitz constant.
\end{rmk}

The proof of Theorem \ref{thm_dimmatch} consists of three steps:
\begin{enumerate}
\vspace{.025in}
\item for polynomials, with generators ${\bf g}$ as variables, their (measurable) differentials are equal to Euclidean gradients;
\vspace{.025in}
\item each Lipschitz function on $X$ can be approximated using PL functions on ${\bf g}(X)$, where ${\bf g}$ generates $\Lip_b(X)$;
\vspace{.025in}
\item the differential of every Lipschitz function exists and agrees with the weak-$*$ (sub)limit of Euclidean gradients.
\end{enumerate}
As usual, for $u : \R^n \to \R$, its $i$th partial derivative is $\partial_iu$ and its gradient is $\nabla u$.

\begin{proof}[Proof of Theorem \ref{thm_dimmatch}]
Without loss, $(X,d)$ is bounded; otherwise we fix $x \in X$, partition $X$ into annuli centered about $x$, and prove the theorem for each annulus.

\vspace{.05in}
{\em Step 1: Euclidean gradients}.\  By hypothesis, there is an $N$-tuple ${\bf g} = (g_j)_{j=1}^N$ that generates $\Lip_b(X)$.  As a shorthand, put $x' = {\bf g}(x)$ for $x \in X$.

We claim that $X$ supports a measurable differentiable structure with a single chart, i.e.\ with $Y = X$ and $\xi = {\bf g}$.  As a first case, for compositions $u = p \circ {\bf g}$ with $p \in P_N$, the smoothness of polynomials on $\R^N$ implies that, for $y \in B(x,r)$,
\begin{equation} \label{eq_diffpoly}
\left.\begin{split}
\hspace{.2in}
&\frac{\big|u(y) - u(x) \;-\; \nabla p(g(x)) \cdot [{\bf g}(y)-{\bf g}(x)] \big|}{r} \\ & 
\hspace{.75in} \;=\;
\frac{|p(y') - p(x') - \nabla p (x') \cdot (y'-x')|}{|y'-x'|} 
\, \frac{|y'-x'|}{r} \\ &
\hspace{.75in} \;\leq\;
\frac{|p(y') - p(x') - \nabla p (x') \cdot (y'-x')|}{|y'-x'|} \, L({\bf g}) \;\longrightarrow\; 0
\hspace{.2in}
\end{split}\right\}
\end{equation}
as $r \to 0$.  So for $\mu$-a.e.\ $x \in X$, Equation \eqref{eq_diffprop} holds with
$$
D_m\big(p \circ {\bf g}\big)(x) \;=\; \nabla p ({\bf g}(x)).
$$
\indent {\em Step 2: PL approximations}.\ For the general case, let ${\bf d} := \{\d_k\}_{k=1}^N$ be the orthogonal basis of $\U(X,\mu)$ from Lemma \ref{lemma_ON}.  Moreover, assume that for each index $i$, the function $|\d_ig_i|$ is $\mu$-a.e.\ bounded away from $0$ and $\infty$, by considering sets
$$
X_k^i\;:=\; \{ x \in X \,:\, 2^{-(k+1)} \,\leq\, |\d_ig_i(x)| \,<\, 2^{-k} \}
$$
and replacing each $\d_i$ with
$(\sum_{k=1}^\infty 2^k \chi_{X_k^i}) \d_i$
as necessary.

The Leibniz rule implies that for all $p \in P_N$ and $\d \in \U(X;\mu)$, we have
\begin{equation} \label{eq_derivpoly}
\d_i(p \circ {\bf g	})(x) \;=\; 
\sum_{k=1}^N \partial_kp(x') \, \d_ig_k(x) \;=\; 
\partial_ip(x') \, \d_ig_i(x).
\end{equation}
Fix $u \in \Lip(X)$.  By hypothesis, there exist $\{p_n\}_{n=1}^\infty \subset P_N$ so that $p_n \circ {\bf g} \wsto u$ in $\Lip(X)$.  In particular, $p_n \circ {\bf g}$ converges locally uniformly to $u$.

Since $X$ is bounded, so is ${\bf g}(X)$.  Let $Q$ be a cube with faces parallel to the coordinate planes and that contains ${\bf g}(X)$.  For each $n \in \N$, let $\{ Q^n_{\bf a} \}_{{\bf a} \in \{1, \cdots, n\}^N}$ be an enumeration of the dyadic subcubes of $Q$ with edge-length $2^{-n}$.

For each ${\bf a}$, let $(Q^n_{\bf a})^0$ be the vertices of $Q^n_{\bf a}$, so $(Q^n_{\bf a})^0 \subset V_N$.  For the restriction $p_n|_{(Q^n_{\bf a})^0}$, let $p_n^{\bf a} : Q^n_{\bf a} \to \R$ be its PL-extension to $Q^n_{\bf a}$, as in Remark \ref{rmk_pl}.  Now define
$$
\ell_n(z) \,:=\, \sum_{\bf a} \chi_{Q^n_{\bf a}}(z) \, p_n^{\bf a}(z).
$$
Clearly $\{\ell_n \circ {\bf g}\}_{n=1}^\infty$ converges locally uniformly to $u$ and that
$$
\sup_n L(\ell_n \circ {\bf g}) \;\leq\; \sup_n L(p_n \circ {\bf g}) \;<\; \infty,
$$
so by weak continuity we have $\d_i(\ell_n \circ {\bf g}) \wsto \d_iu$ in $L^\infty(\R^N,{\bf g}_\#\mu)$.

Note that $\ell_n$ is smooth off of a locally-finite union of lower-dimensional simplices, so by Corollary \ref{cor_planezero}, it is differentiable ${\bf g}_\#\mu$-a.e.\ in $Q^n_{\bf a}$.  It also satisfies
\begin{equation} \label{eq_diffzero}
\ell_n(w) - \ell_n(z) - \nabla\ell_n(z) \cdot [w-z] \;=\; 0
\end{equation}
for ${\bf g}_\#\mu$-a.e.\ $z \in Q_{\bf a}^n$ and for all $w$ sufficiently close to $z$; more precisely, it suffices that $z$ and $w$ lie in the same simplex in the triangulation of $\R^N$ at scale $2^{-n}$.

\vspace{.05in}
{\em Step 3: Weak and weak-$*$ sublimits}.\  
Recall that $\ell_n$ is linear on sub-simplices of $Q^n_{\bf a}$, so Equation \eqref{eq_derivpoly} and the locality property imply that
\begin{equation} \label{eq_matchdiff}
\d_i(\ell_n \circ {\bf g})(x) \;=\; \partial_i\ell_n(x') \, \d_ig_i(x)
\end{equation}
holds for ${\bf g}_\#\mu$-a.e.\ $x \in {\bf g}^{-1}(Q_{\bf a}^n)$.
Since $|\d_ig_i|$ is $\mu$-a.e.\ bounded away from zero, Equation \eqref{eq_matchdiff} and Lemma \ref{cor_planezero} imply that $\{\partial_i\ell_n\}_{n=1}^\infty$ is bounded in $L^\infty(\R^N;{\bf g}_\#\mu)$, for each $1 \leq i \leq N$.  By the Banach-Alaoglu theorem, there is a weak-$*$ limit $U_i$ for some subsequence $\{\partial_i\ell_{n_j}\}_{j=1}^\infty$.

The image ${\bf g}(X)$ has finite ${\bf g}_\#\mu$-measure, so 
$L^q(\R^N,{\bf g}_\#\mu) \subset L^1(\R^N,{\bf g}_\#\mu)$
holds and each $U_i$ is a weak limit of $\{\partial_i\ell_{n_j}\}_{j=1}^\infty$ in $L^q(\R^N,{\bf g}_\#\mu)$ as well.  This function space is reflexive, so by Mazur's lemma there are finite convex combinations 
$$
l_m \;:=\; \sum_{j=m}^{M(m)} \lambda_{mj} \ell_{n_j}
$$
whose partial derivatives $\{\partial_il_m\}_{m=1}^\infty$ converge in $L^q$-norm to $U_i$, and hence a further subsequence converges ${\bf g}_\#\mu$-a.e.\ to $U_i$.

By repeating a similar argument on further subsequences and re-indexing, we may assume that $\{\d_i(l_m \circ {\bf g})\}_{m=1}^\infty$ also converges pointwise $\mu$-a.e.\ to $\d_iu$ as well.

Now put $U := (U_1,\cdots,U_n)$, let $\e > 0$ be given, and let $K$ be the constant from \eqref{eq_lipderiv}.  Choose $m = m(x,\e) \in \N$ so that following inequalities hold:
\begin{equation} \label{eq_smallderiv}
|U(x') - \nabla l_m(x')| \;<\; \frac{\e}{2L({\bf g})} \, \text{ and } \,
\sum_{i=1}^n \big|\d_i[u - l_m \circ {\bf g}](x)\big| \;<\; \frac{\e}{4K}.
\end{equation}
In particular, if $\Lip[u - l_m \circ {\bf g}](x) = 0$ holds for all but finitely many indices $m$, then for sufficiently small $r = r(\e,m,x) > 0$, we have
$$
\frac{\left|\big[u - l_m \circ {\bf g}\big](y) - \big[u - l_m \circ {\bf g}\big](x)\right|}{r} \;\leq\;
\frac{\e}{2}
$$
Otherwise, Equation \eqref{eq_smallderiv} and the Lip-derivation inequality \eqref{eq_lipderiv} imply that an analogous choice $r = r(\e,m,x) > 0$ leads to a similar estimate, for $y \in B(x,r)$:
\begin{eqnarray*}
\frac{\left|\big[u - l_m \circ {\bf g}\big](y) - \big[u - l_m \circ {\bf g}\big](x)\right|}{r} &\leq&
2 \Lip[u - l_m \circ {\bf g}](x) \\ &\leq&
2K \sum_{i=1}^n \big|\d_i[u - l_m \circ {\bf g}](x)\big| \;<\;
\frac{\e}{2}.
\end{eqnarray*}
Since $m \in \N$ is now fixed, take $r > 0$ smaller as necessary so that $x'$ and $y'$ lie in the same simplex with respect to the fixed triangulation of $\R^N$ at scale $2^{-m}$.  Equation \eqref{eq_diffzero} applies to $x',y' \in \R^N$, so from this and the above inequalities, we obtain
\begin{equation*} 
\begin{split}
&\frac{\big|u(y) - u(x) - U\big(x'\big) \cdot [y' - x']\big|}{d(x,y)} \\ &
\hspace{.1in} \;\leq\;
\frac{\big|u(y) - u(x) - [l_m(y') - l_m(x')]\big|}{d(x,y)} \;+\;
\frac{\big|l_m(y') - l_m(x') - \nabla l_m(x') \cdot [y' - x']\big|}{d(x,y)} \\ & 
\hspace{2.34in} \;+\;
\big|U(x') - \nabla\l_m(x')\big| \, \frac{|{\bf g}(y)-{\bf g}(x)|}{d(x,y)} \\ &
\hspace{.1in} \;\leq\;
\frac{\left|\big[u - l_m \circ {\bf g}\big](y) - \big[u - l_m \circ {\bf g}\big](x)\right|}{d(x,y)}
\;+\; 0 \;+\; 
L({\bf g}) \, \big|U(x') - \nabla\l_m(x')\big| \\ &
\hspace{.1in} \;<\; \frac{\e}{2} + \frac{\e}{2} \;=\; \e
\end{split}
\end{equation*}
whenever $x \in {\bf g}^{-1}(Q_{\bf a}^n)$, so Equation \eqref{eq_diffprop} therefore follows with $Du = U \circ {\bf g}$.

By construction, $\partial_i\ell_n$ and $\partial_il_m$ have the same weak-$*$ limits in $L^\infty(\R^N,{\bf g}_\#\mu)$, so from Equation \eqref{eq_matchdiff} and the definition of ${\bf g}_\#\mu$, the differential becomes
$$
\hspace{1.5in}
Du \;=\; (\d_1g_1)^{-1}{\bf d}u.
\hspace{1.5in} \qedhere
$$
\end{proof}

The general case follows a similar idea.  If $M = {\rm rank}[\U(X,\mu)]$ is strictly smaller than the number of generators for $\Lip_b(X)$, then by applying a local ``change of variables,'' appropriate subsets of $M$ generators can nonetheless be used as coordinates for a measurable differentiable structure.

For the sake of clarity, the argument is again divided into several steps: one handles PL approximations of Lipschitz functions, and the other gives the explicit change-of-variables technique.

\begin{proof}[Proof of Theorem \ref{thm_fingen-measdiff}]
As given in Lemma \ref{lemma_ON}, let ${\bf d} := \{\d_i^*\}_{i=1}^M$ be an orthogonal basis of $\U(X,\mu)$ and $\{X_l\}_{l=1}^L$ a measurable partition of $X$.  It suffices to construct a measurable differentiable structure on each $X_l$, so without loss we will suppress the index $l$ and write $X = X_l$ and $f_j = f^l_j$, etc.

Up to reindexing, assume that $g_j = f_j$, for $1 \leq j \leq M$, and write the tuples as 
$$
{\bf f} \;:=\; (f_j)_{j=1}^M : X \to \R^M \, \text{ and } \, 
{\bf g} \;:=\; (g_j)_{j=1}^N : X \to \R^N.
$$

{\em Step 1:\ Change of Variables}.\ Fix a composition $u = p \circ {\bf g}$, for $p \in P_N$.  By \eqref{eq_diffpoly}, $u$ satisfies \eqref{eq_diffprop} as before, with measurable differential $Du = (\nabla p) \circ {\bf g}$.

Since $\mu$ satisfies the Lebesgue differentiation property \eqref{eq_leb}, it follows that $\mu$-almost every $x \in X$ is a density point of the matrix-valued functions 
$$
{\bf d g}(x) \;=\; 
[\d_i g_j(x)]_{i=1,j=1}^{M,N} \, \text{ and } \,
{\bf d}(p \circ {\bf g})(x) \;=\; 
[\d_i(p \circ g_j)(x)]_{i=1,j=1}^{M,N}.
$$
Fixing such a point $x_0 \in X$, define a linear map $T = (T_1, \cdots, T_N)$ on $\R^N$ by
\begin{equation} \label{eq_changecoords}
T_j(z_1, \cdots, z_N) \;:=\; 
\begin{cases}
z_j, & \text{if } j \leq M \\
\d_1g_1(x_0)\, z_j - \displaystyle{ \sum_{i=1}^M \d_ig_j(x_0) } \, z_i, & \text{if } M < j \leq N.
\end{cases}
\end{equation}
Since $T$ is invertible by Lemma \ref{lemma_ON}, the tuple $\g := T \circ {\bf g}$ also generates $\Lip_b(X)$.  So with $\p := p \circ T^{-1}$, the same fixed Lipschitz function 
$$
u \;=\; p \circ {\bf g} \;=\; \p \circ \g
$$
satisfies \eqref{eq_diffprop} with differential
$x \mapsto \nabla\p(\g(x))$
and with coordinates $\g$ on $X$.  However, at $x = x_0$ the matrix representation for ${\bf d}\g$ is
\begin{equation} \label{eq_derivjacobi}
{\bf d}\g(x_0) \;=\; [\d_i\g_j(x_0)]_{i=1,j=1}^{M,N} \;=\;
\d_1g_1(x_0) \begin{bmatrix}
I_M \\
O
\end{bmatrix}
\end{equation}
where $O$ is the $M \times (N-M)$ zero matrix and $I_M$ is the $M \times M$ identity matrix.  In particular, ${\bf d}\g_j(x_0) = 0$ holds for $j > M$, so \eqref{eq_lipderiv} implies
$$
\Lip[\g_j](x_0) \;=\; 0.
$$
Putting $\nabla^|p := (\partial_1p, \cdots, \partial_Mp)$ for the truncated gradient, we compute
\begin{equation*}
\begin{split}
&
\frac{u(y) - u(x_0) - \nabla^|p(\g(x_0)) \cdot [{\bf f}(y) - {\bf f}(x_0)]}{d(x_0,y)} \\ &
\,=\;
\frac{u(y) - u(x_0) - \nabla\p(\g(x_0)) \cdot [\g(y) - \g(x_0)]}{d(x_0,y)} \,+\,
\sum_{j=M+1}^N \partial_jp(\g(x_0)) \frac{\g_j(y) - \g_j(x_0)}{d(x_0,y)}
\end{split}
\end{equation*}
where the first step follows from $\p_j = p_j$ and $f_j = \g_j$, for $1 \leq j \leq M$.
Taking limes superior, the previous identities imply that
\begin{equation*}
\begin{split}
&
\limsup_{y \to x_0} \frac{\big|u(y) - u(x_0) - \nabla^|p(\g(x_0)) \cdot [{\bf f}(y) - {\bf f}(x_0)]\big|}{d(x_0,y)} \\ &
\hspace{.3in} \;=\;
\Lip\big[u - \nabla^|p(\g(x_0)) \cdot {\bf f}\big](x_0) \\ &
\hspace{.3in} \;\leq\;
\Lip\big[u - \nabla\p(\g(x_0)) \cdot \g\big](x_0) \,+\,
\sum_{j=M+1}^N |\partial_jp(\g(x_0))| \, \Lip\big[\g_j(x_0)\big] \;=\; 0 + 0.
\end{split}
\end{equation*}
To summarize, at $x =x_0$  the vectorfield $x \mapsto \nabla^|p (\g(x))$ satisfies the role of the differential under lower-dimensional coordinates ${\bf f}$, with
$$
Du(x_0) \;=\;
D(p \circ {\bf g})(x_0) \;=\;
\nabla^|p (\g(x_0)).
$$

\vspace{.05in}
{\em Step 2:\ PL approximations}.\ We now sketch an argument similar to Step 2 of Theorem \ref{thm_dimmatch}.
Briefly, every $u \in \Lip_b(X)$ can be weak-$*$ approximated by a sequence $\{\ell_k \circ g \}_{k=1}^\infty$, where each $\ell_k : \R^N \to \R$ is piecewise-linear.  This implies that 
$$
\d_i(\ell_k \circ g) \wsto \d_iu
$$
in $L^\infty(X;\mu)$, for $1 \leq i \leq M$.  By a Mazur-type argument as before, we may assume that the convergence is $\mu$-a.e.\ pointwise.  Fixing such a point $x_0 \in X$ and with $T$ as in \eqref{eq_changecoords}, the approximants also fit a change of variables of the form
$$
\mathfrak{l}_k \circ \g \;:=\; (\ell_k \circ T^{-1}) \circ (T \circ g) \;=\; \ell_k \circ {\bf g},
$$
so using \eqref{eq_derivjacobi}, the pointwise convergence can be rewritten as
\begin{eqnarray*}
\d_1f_1(x_0) \, \nabla^|\ell_k(\g(x_0)) &=&
\nabla\mathfrak{l}_k(\g(x_0)) \cdot {\bf d}\g(x_0) \\ &=&
{\bf d}(\mathfrak{l}_k \circ \g)(x_0) \;=\;
{\bf d}(\ell_k \circ {\bf g})(x_0) \;\to\; {\bf d}u(x_0)
\end{eqnarray*}
and \eqref{eq_diffprop} follows similarly as in Step 3 of Theorem \ref{thm_dimmatch}, with differential 
$$
Du(x_0) \;:=\; [\d_1f_1(x_0)]^{-1} {\bf d}u(x_0). 
$$
Note that if condition (2) holds on $X$, then there must exist $f \in \Lip(X)$ so that $\Lip[f]$ is positive on a set of positive $\mu$-measure.  It follows that any measurable differentiable structure on $X$ must be non-degenerate.
\end{proof}

\section{The Case of Doubling Measures} \label{sect_doubling}

We begin with a few useful facts about doubling measures on metric spaces, and then proceed with the proof of Theorem \ref{thm_doublingdiff}.

\subsection{Doubling metric spaces}

Recall that if $\mu$ is $\kappa$-doubling on $(X,d)$, then $X$ is {\em (metrically) $N$-doubling} with $N = N(\kappa) \in \N$: this means that every ball $B(x,r)$ in $X$ can be covered by at most $N$ balls with radius $r/2$ and centers in $B(x,r)$.  Moreover, this geometric condition gives rise to good approximation properties for Lipschitz functions, just as in the case of $\R^n$.  Using these approximations, the proof ideas of Theorems \ref{thm_dimmatch} and \ref{thm_fingen-measdiff} naturally extend to the metric space setting.

Indeed, by taking $\e$-nets\footnote{The notion of an $\e$-net from metric geometry should not be confused with nets, as in Remark \ref{rmk_net}, which are generalised sequences that detect convergence.} on a doubling metric space, one may construct analogues of piecewise-linear approximations of Lipschitz functions, in the weak-$*$ sense of Lemma \ref{lemma_dualbanach}.  Note that similar techniques have been used before by Semmes \cite[Eq.\ B.6.24]{Gromov:book} and Keith \cite[Defn 4.1]{Keith-distvectors} but with different applications.

To begin, recall that on a metric space $(X,d)$, a {\em (maximally separated) $\e$-net}, for $\e > 0$, is a subset $[X]_\e \subset X$ with the property that, for some $C \geq 1$, 
\begin{itemize}
\item every $x \in X$ satisfies $d(x,x') \leq C\e$, for some $x' \in [X]_\e$;
\item if $x',x'' \in [X]_\e$ with $x' \neq x''$, then $d(x',x'') \geq \e$.
\end{itemize}

Such $\e$-nets always exist for doubling spaces \cite{Christ}, \cite[Lemma B.7.3]{Gromov:book}.

\begin{defn} 
Let $(X,d)$ be a metric space and let $\e > 0$.  If $[X]_\e$ is an $\e$-net of $X$, then for $u \in \Lip(X)$, the function
$$
[u]_\e(x) \;:=\; \inf \big\{ u(x') + L(u) \cdot d(x,x') \,:\, x' \in [X]_\e \big\}
$$
is called the {\em piecewise-distance approximation of $u$ (with respect to $[X]_\e$)}.
\end{defn}

\begin{rmk} \label{rmk_gengenerators}
It is clear that the approximations $\{u_\e\}_{\e > 0}$ are each $L(u)$-Lipschitz and converge uniformly to $u$, as $\e \to 0$.  As a consequence, distance functions 
$$
d_{x'}(x) \;:=\; d(x,x'), \, \text{ for } \, x' \in [X]_\e
$$
form generating sets for $\Lip_b(X)$, in a generalised sense.  Moreover, $\d u$ is the weak-$*$ limit of (locally finite) sums of ${\d}d_{x'}$, for each $\d \in \U(X,\mu)$.  The proof of Lemma \ref{lemma_ON} therefore applies to linearly independent sets in $\U(X,\mu)$ in this setting; compare with \cite[Rmk 7.3.2]{Keith} and \cite[Cor 6.28]{Schioppa}.
\end{rmk}

\subsection{Rank Bounds for Derivations}

As discussed in \S1, not every doubling metric space admits a bi-Lipschitz embedding into some $\R^n$. Assouad's embedding theorem \cite{Assouad} asserts, however, that a weaker statement holds true.  The formulation below is due to Naor and Neiman \cite{Naor:Neiman}, where the embedding dimension is independent of ``snowflaking.''

\begin{thm}[Assouad, Naor-Neiman] \label{thm_assouad}
Let $(X,d)$ be $N$-doubling for some $N \in \N$.  For each $s \in (0,1)$, there is an embedding $\zeta : X \to \R^n$ so that
$$
K^{-1}\,d(x,y)^s \;\leq\;
|\zeta(y) - \zeta(x)| \;\leq\;
K\,d(x,y)^s
$$
holds, for all $x,y \in X$.  Here $n = n(N) \in \N$ and $K = K(s,N) \geq 1$.
\end{thm}

Similarly to the case of spaces $(X,d,\mu)$ with finitely-generated Lipschitz algebras, the doubling condition gives rise to an upper bound for the rank of $\U(X,\mu)$.  We formulate this below as a quantitative version of Lemma \ref{lemma_doublingbd}.

\begin{lemma} \label{lemma_doublingbd2}
Let $N \in \N$.  If $(X,d)$ is a $N$-doubling metric space, then there exists $M = M(N) \in \N$ so that for every Radon measure $\mu$ on $X$, the module $\U(X,\mu)$ has rank at most $M$.
\end{lemma}

\begin{rmk}
The bound $M$ is not sharp in general.  Indeed, the proof gives $M = n$, where $n$ is the target dimension of the Assouad embedding. 
In contrast, Carnot groups are doubling, yet their modules of derivations (with respect to Haar measure) have rank strictly less than $M$ \cite[Thm 39]{WeaverED}.
\end{rmk}

The idea of Lemma \ref{lemma_doublingbd2} is simple but the proof is technical.  For clarity, it is divided into three steps:
\begin{enumerate}
\item By taking piecewise-distance approximations $[\zeta]_\e$ of the Assouad embedding $\zeta$, when restricted to an $\e$-net, derivations on $X$ have well-defined pushforwards on $\R^n$.  

\vspace{.05in}
\item In general, the family $\{[\zeta]_\e\}_{\e > 0}$ is not uniformly Lipschitz.  The nontrivial step is in showing that suitable {\em composite} approximations satisfy 
$$
[u \circ \zeta^{-1}]_\e \circ [\zeta]_\e \;\to\; u
$$
for all Lipschitz functions $u$ on $X$ and in particular, for all generators of $\Lip_b(X)$.
For technical reasons, however, the argument is localised so that the above convergence is applied to points of density.

\vspace{.05in}
\item Since any collection of $n+1$ derivations on $\R^n$ is linearly dependent, so are the pushforwards of the original derivations on $X$.
Using the Chain Rule and Lemma \ref{lemma_pushfwd}, one shows that the corresponding Jacobians on $X$ must vanish, which contradicts Corollary \ref{cor_nonsing}.
\end{enumerate}

\begin{proof}[Proof of Lemma \ref{lemma_doublingbd2} (and \ref{lemma_doublingbd})]
Since $\mu$ is doubling, we may assume by the locality property that $X$ is bounded.  Fix $s \in (0,1)$, let $\zeta : X \to \R^n$ be Assouad's embedding, and put $Y = \zeta(X)$.

\vspace{.025in}
{\em Step 1:\ Piecewise-distance approximations}.\  For each $\e > 0$, let $[X]_\e := \{x_i\}_{i=1}^\infty$ be an  $\e$-net of $X$.  By Theorem \ref{thm_assouad}, the image of $[X]_\e$, denoted by
$$
[Y]_\e \;:=\; \zeta([X]_\e),
$$
is also an $K\e^s$-net of $Y$, for some $K \geq 1$.  Moreover, the restriction $\zeta|_{[X]_\e}$ satisfies
$$
|\zeta(x_i) - \zeta(x_j)| \;\leq\;
K \, d(x_i,x_j)^{s-1}d(x_i,x_j) \;\leq\;
K\e^{s-1} d(x_i,x_j)
$$
for all $i,j \in \N$.  The piecewise-distance approximation $[\zeta]_\e$ of (the components of) this restriction is therefore $\sqrt{n}K\e^{s-1}$-Lipschitz.

Note that $\zeta^{-1} : Y \to X$ is also locally Lipschitz, in that Theorem \ref{thm_assouad} implies
$$
L(\zeta^{-1}|_A) \;\leq\;
K^\frac{1}{s} \diam(\zeta^{-1}(A))^{1-s}
$$
for all (bounded) subsets $A \subset Y$.
For each $x_0 \in X$, put $B_\e := B(x_0,\e)$.  For each $u \in \Lip(X)$, consider the restriction of $u \circ \zeta^{-1} : Y \to \R$ to $\zeta(B_\e) \cap [Y]_\e$ and let 
$$
[ u \circ \zeta^{-1} ]_{B_\e} : \zeta(B_\e) \to \R
$$
be its piecewise-distance approximation on $\zeta(B_\e)$, which is Lipschitz with constant
$$
L\left( [ u \circ \zeta^{-1} ]_{B_\e} \right) \;\leq\;
L(u) \, L\big( \zeta^{-1}\big|_{\zeta(B_\e)} \big) \;\leq\;
L(u)K^\frac{1}{s}(2\e)^{1-s}.
$$
This means that
$\{ [ u \circ \zeta^{-1}]_{B_\e} \circ [\zeta]_\e \}_{\e > 0}$
is uniformly $K'$-Lipschitz, with $K' \approx L(u)$ independent of $\e$.  Now define auxiliary functions $\tilde{u}_\e : (X \setminus B_{2\e}) \cup B_\e \to \R$ by
\begin{equation} \label{eq_localapprox}
\tilde{u}_\e \;:=\;
u \chi_{X \setminus B_{2\e}} \,+\, \big( [u \circ \zeta^{-1}]_{B_\e} \circ [\zeta]_\e \big) \chi_{B_\e}.
\end{equation}
We claim that $\{ \tilde{u}_\e \}_{\e > 0}$ is also uniformly Lipschitz, relative to their domains of definition.  It suffices to check pairs $x \in X \setminus B_{2\e}$ and $x' \in B_\e$, for each $\e > 0$, so 
$$
0 \;<\; d(x_0,x') \;\leq\; \e \;\leq\; d(x,x').
$$
Keeping in mind that $x_0 \in X$ satisfies $u(x_0) = \tilde{u}_\e(x_0)$,
it follows that
\begin{eqnarray*}
\frac{| \tilde{u}_\e(x) - \tilde{u}_\e(x') |}{d(x,x')} &=&
\frac{|u(x) - \tilde{u}_\e(x')|}{d(x,x')} \;\leq\;
\frac{|u(x) - u(x_0)|}{d(x,x')} \,+\, \frac{\big|\tilde{u}_\e(x_0) - \tilde{u}_\e(x')\big|}{\e} \\ &\leq&
\frac{|u(x) - u(x_0)|}{d(x,x_0)} \,+\, L\big(\tilde{u}_\e\big|_{B_\e}\big) \\ &\leq&
L(u) \,+\, L\Big( [u \circ \zeta^{-1}]_{B_\e} \circ [\zeta]_\e\Big|_{B_\e} \Big) \;\leq\;
L(u) \,+\, K'.
\end{eqnarray*}
The claim now settled, let $u_\e : X \to \R$ be the McShane extension of $\tilde{u}_\e$.  By the previous argument, it follows that $\{u_\e\}_{\e > 0}$ is uniformly Lipschitz and converges uniformly to $u$, so $u_\e \wsto u$ in $\Lip_b(X)$.

\vspace{.05in}
{\em Step 2:\ Embeddings and Jacobians}.\
Since $X$ is bounded, each $\e$-net of $X$ from before becomes a finite set $[X]_\e = \{x_i\}_{i=1}^{N(\e)}$.  Taking distance functions
$$
g_i(x) \;:=\; d(x_i,x),
$$
the conclusions of Corollary \ref{cor_nonsing} and Lemma \ref{lemma_ON} apply to the set ${\bf g} := \big\{(g_i)_{i=1}^{N(\e)}\big\}_{\e>0}$, as indicated before in Remark \ref{rmk_gengenerators}.

Let ${\bf d} := \{\d_i^*\}_{i=1}^M$ be an orthogonal basis of $\U(X,\mu)$ with respect to ${\bf g}$ and let $\{X_l\}_{l=1}^\infty$ be the associated partition of $X$.  For clarity, we suppress the symbols $*$ and $l$, so $\d_i = \d_i^*$, and ${\bf f} := \{f^j\}_{j=1}^M$ denotes the subset of ${\bf g}$ on $X := X_l$, from Lemma \ref{lemma_ON}.  In particular, property (3) of that lemma implies the $\mu$-a.e.\ identity
\begin{equation} \label{eq_jacobianid}
\det({\bf d f}) \;=\; (\d_1f_1)^M \;>\; 0,
\end{equation}
where if necessary, $\d_1$ is replaced with $(\chi_{\{\d_1f_1 > 0\}} - \chi_{\{\d_1f_1 < 0\}})\d_1$.

Fix a sequence of scales $\e = 2^{-\a}$, for $\a \in \N$.  As in Formula \eqref{eq_localapprox} in Step 1, for $u = f^j$ consider analogous sequences of functions
$$
\tilde{f}^j_\a \;:=\;
f^j \chi_{X \setminus B_{2\e}} \,+\,
\big( [f^j \circ \zeta^{-1}]_{B_\e} \circ [\zeta]_{2^{-\a}} \big) \chi_{B_\e}
$$
and let $f^j_\a$ denote the McShane extension of $\tilde{f}^j_\a$.  Clearly $\{f^j_\a\}_{\a=1}^\infty$ converges weak-$*$ to $f^j$ in $\Lip_b(X)$, so $\d_if^j_\a \wsto \d_if^j$ in $L^\infty(X;\mu)$ for each $1 \leq i,j \leq M$.

Since $\mu(X) < \infty$, for every $q \in (1,\infty)$ we have $\d_if^j_\a \rightharpoonup \d_if^j$ (i.e.\ weakly) in $L^q(X;\mu)$.  So by a Mazur-type argument as before, there exist convex combinations 
$$
\f^j_\a \,:=\, \sum_{\b=\a}^\infty c_{\a\b}f^j_\b \, \textrm{ and } \,
\f_\e \;:=\; \f_\a \,:=\, (\f^1_\a, \cdots, \f^M_\a)
$$
where $\d_i\f_\a^j \to \d_if^j$ holds $\mu$-a.e.\ on $X$, as well as
$\det( {\bf d}\f_\a ) \to \det( {\bf df} )$.
Equation \eqref{eq_jacobianid} therefore implies that, for points $x \in X$ of $\mu$-density for $\det({\bf d}\f_\a)$, we have
\begin{equation} \label{eq_jacobianpos}
0 \;<\;
\det( {\bf d}\f_\a(x) ) \;=\;
\lim_{r \to 0} \frac{1}{\mu(B(x,r))}\int_{B(x,r)} \det{\bf d}\f_\a \, d\mu
\end{equation}
whenever $\a$ is sufficiently large.  In particular, this applies to points in $X$ that are $\mu$-density points of $\d_ig_j$, for all $i$ and $j$ simultaneously.

\vspace{.05in}
{\em Step 3:\ Factoring Jacobians via pushforwards}.\
Since the measurable functions
$$
\big\{ \{\d_ig_j\}_{i=1,j=1}^{M,N(2^{-\a})} \big\}_{\a=1}^\infty
$$
form a countable set, the intersection of their $\mu$-density points has full measure in $X$; fix such a point $x=x_0$.

Towards a contradiction, suppose that $M > n$.  For $\mu^\e := ([\zeta]_\e)_\#\mu$, consider pushforward derivations $\d_i^\e := ([\zeta]_\e)_\#\d_i $ in $\U(\R^n,\mu^\e)$.  Lemma \ref{lemma_pushfwd} then gives
\begin{equation} \label{eq_pushfwdconv}
\int_{\zeta(B_\e)} \psi \, \d_i^\e[\f^j \circ \zeta^{-1}]_{B_\e} d\mu^\e \;=\;
\int_{B_\e} (\psi \circ [\zeta]_\e) \, \d_i\f_\e^j \, d\mu 
\end{equation}
for all $\psi \in C^0_c(\R^n) \subset L^1(\R^n,\mu^\e)$ and all indices $1 \leq i,j \leq M$.  

More generally, for each positive function $w \in L^\infty(X)$ the measure $d\mu_w := w d\mu$ is mutually absolutely continuous with $\mu$, so as modules, 
$$
\U(X,\mu) \;\cong\; \U(X,\mu_w).
$$
So with the same pushforwards $\d_i^\e$ as before, Equations \eqref{eq_pushfwd} and \eqref{eq_pushfwdconv} give
\begin{eqnarray*}
\int_{B_\e} (\psi \circ [\zeta]_\e) ( \big( \d_i[\f^j \circ \zeta^{-1}]_{B_\e} \big) \circ [\zeta]_\e ) \, w \, d\mu &=&
\int_{\zeta(B_\e)} \psi \, \d_i^\e[\f^j \circ \zeta^{-1}]_{B_\e} \, d( ([\zeta]_\e)_\#\mu_w ) \\ &=&
\int_{B_\e} (\psi \circ [\zeta]_\e) \, \d_i\f^j_\e \, w \, d\mu.
\end{eqnarray*}
In particular, this holds for each $\a \in \N$. Taking $w$ to be sums of products of entries of ${\bf d}\f_\a$, we further obtain, for all $\psi \in L^1(\R^n,\mu^\e)$, the identity
\begin{equation} \label{eq_jacobian}
\int_{B_\e} (\psi \circ [\zeta]_\e) \det( {\bf d}[\f \circ \zeta^{-1}]_{B_\e} \circ [\zeta]_\e ) \, d\mu \;=\;
\int_{B_\e} (\psi \circ [\zeta]_\e) \, \det({\bf d}\f_\e) \, d\mu.
\end{equation}
Since $\R^n$ is $n$-generated and $M > n$, Lemma \ref{lemma_fingenbd} implies that ${\bf d}^\e := (\d_i^\e)_{i=1}^M$ is linearly dependent in $\U(\R^n,\mu^\e)$, so at $\mu^\e$-a.e.\ point of $\zeta(B_\e)$, the $n \times M$ matrix
$$
{\bf d^\e x} \;:=\; [ \d_i^\e x_j ]_{i=1,j=1}^{M,n}
$$
has (matrix) rank at most $n$.  The Chain Rule further implies that
$$
{\bf d}^\e{\bf h} \;=\; v_{\bf h} \cdot {\bf d}^\e{\bf x}
$$
holds for all Lipschitz maps ${\bf h} = (h_1, \cdots, h_M): \R^n \to \R^M$ with associated $M$-tuples of vectorfields $v_{\bf h} := (v_{h_1}, \cdots, v_{h_M})$ from Lemma \ref{lemma_chainrule}.  As a consequence, the $M \times M$ matrix ${\bf d}^\e{\bf h}$ also has rank at most $n$, so for $\mu^\e$-a.e.\ point in $\zeta(B_\e)$, 
\begin{equation} \label{eq_detzero}
\det( {\bf d}^\e{\bf h} ) \;=\; 0.
\end{equation}
With $\e = 2^{-\a}$ and the same scalars for convex combinations as before, put 
$$
{\bf h}_\a = \sum_{\b=\a}^\infty c_{\a\b} \, [ {\bf f} \circ \zeta^{-1}]_{B_\e}, \, \text{ hence } \,
{\bf h}_\a \circ [\zeta]_\e \;=\; \f_\a,
$$
and let $\lambda_i^\a$ denote the determinant of the cofactor of ${\bf d}^\e{\bf h}_\a$, accounting for sign, with the first row and $i$th column omitted.

So for all $\a \in \N$, Equations \eqref{eq_pushfwdconv}-\eqref{eq_detzero} imply that each $r \in (0,2^{-\a})$ and each non-negative $\psi \in C^0_c(\R^n)$, with $\psi|_{\zeta(B(x,r))} > 0$, satisfy the identities
\begin{eqnarray*}
0 \,=\,
\int_{\zeta(B(x,r))} \psi \det({\bf d}^\e{\bf h}_\a) \,d\mu^\e &=&
\sum_{i=1}^M \int_{\zeta(B(x,r))} \psi \, \lambda_i^\a \big[([\zeta]_\e)_\#\d_i\big]h_\a^1 \,d\mu^\e \\ &=&
\int_{B(x,r)} (\psi \circ [\zeta]_\e) \sum_{i=1}^M (\lambda_i^\a \circ [\zeta]_\e) \d_i\f_\a^1 \, \,d\mu \\ &=&
\int_{B(x,r)} (\psi \circ [\zeta]_\e) \det({\bf d}\f_\a) \,d\mu
\end{eqnarray*}
Letting $r \to 0$, this contradicts \eqref{eq_jacobianpos} and proves the lemma.
\end{proof}

\subsection{Derivations induce differentiability}

We now show that measurable differentiable structures exist on spaces that support a doubling measure and satisfy the Lip-derivation inequality.  The proof reduces to Lemma \ref{lemma_doublingbd2} in a similar way as how the proof of Theorem \ref{thm_fingen-measdiff} reduces to Theorem \ref{thm_dimmatch}.  We briefly sketch the idea.

\begin{proof}[Proof of $(2) \Rightarrow (1)$ for Theorem \ref{thm_doublingdiff}]
Assume all the notation from the proof of Lemma \ref{lemma_doublingbd2}.  Since $\mu$ is doubling on $X$, there exists $M = M(\mu) \in \N$ so that $\U(X,\mu)$ has rank at most $M$.

Once again, let $\e > 0$ and fix an $\e$-net $[X]_\e = \{x_k\}_{k=1}^\infty$ of $X$.
Recall that every $u \in \Lip(X)$ can be weak-$*$ approximated by McShane extensions of functions
$$
g^k(x) \;:=\; d(x,x_k).
$$
By Lemma \ref{lemma_ON} and Remark \ref{rmk_gengenerators} there is a basis ${\bf d} := \{\d_i\}_{i=1}^M$ of $\U(X,\mu)$ and a measurable partition $\{X_l\}$ of $X$, so that on each $X_l$, the basis is orthogonal to some subset ${\bf f} = (f^j)_{j=1}^M$ of ${\bf g} := \{g^k\}_{k=1}^\infty$.  The remaining generators are denoted ${\bf h} := \{h^k\}_{k=M+1}^\infty$.

By Lemma \ref{lemma_ON}, the points in $X$ that are $\mu$-density points of $\d_ih^k$, for every $k \in \N$, 
forms a subset in $X$ whose complement has zero $\mu$-measure.
Let $x_0 \in X$ be such a point and without loss, assume $\d_if^i(x_0) > 0$.  Similarly as in Step 1 of Theorem \ref{thm_fingen-measdiff}, define $\h^k \in \Lip_b(X)$ as
$$
\h^k \;:=\; \d_1f^1(x_0) h^k \;-\; \sum_{j=1}^M \d_jh^k(x_0) \, f^j.
$$
Since ${\bf g}$ generates $\Lip(X)$, so does the collection of Lipschitz functions
$$
\g^k \;:=\;
\begin{cases}
f^k, & \text{if } k \leq M \\
\h^k, & \text{if } k > M.
\end{cases}
$$
Arguing once more by \eqref{eq_diffpoly}, all functions $u = p \circ {\bf f}$, for $p \in P_M$, are differentiable with respect to ${\bf f}$, with differential given by $Du = (\nabla p) \circ {\bf f}$.

By construction, however, it follows that ${\bf d}\h^k(x_0) = 0$, so $\Lip[\h^k](x_0) = 0$ further follows from the Lip-derivation inequality \eqref{eq_lipderiv}.  Thus every $h^k$ can be differentiated with respect to coordinates ${\bf f}$; the same holds for all compositions of $N$-tuples of ${\bf g}$ with polynomials in $P_N$, for all $N \in \N$.

Now let $u \in \Lip(X)$ be given; omitting a set of $\mu$-measure zero if necessary, assume $|{\bf d}u(x_0)| \neq \infty$.  As in the proof of Lemma \ref{lemma_doublingbd2} with $\e = 2^{-\a}$ for $\a \in \N$, the approximations $u_\a := u_{2^{-\a}}$ are uniformly Lipschitz and converge pointwise to $u$.  By a Mazur-type argument as before, there exist convex combinations 
$$
v_l \;:=\;
\sum_\a \lambda_{\a l} \, u_\a
$$
so that $\d_iv_l \rightharpoonup \d_iu$ in $L^q(X,\mu)$, for some $q \in (1,\infty)$.  
So by taking further subsequences and omitting another set of $\mu$-measure zero if necessary, assume further that $\d_iv_\b(x_0) \to \d_iu(x_0)$.

From the orthogonality of ${\bf d}$ and the fact that ${\bf d}\h^k(x_0) = 0$, we observe
$$
\d_iv_l(x_0) \;=\;
\sum_\a \lambda_{kl}(x_0) \, \d_iu_\a(x_0) \;=\;
\sum_\a \lambda_{kl}(x_0) \, \d_i\g^{\b_\a}(x_0) \;=\;
\lambda_{il}(x_0) \d_if^i(x_0).
$$
Since $\d_if^i(x_0) > 0$, the sequence $\{\lambda_{il}\}_{l=1}^\infty$ converges to some $\lambda_i \in \R$.

We now proceed as in Step 3 of Theorem \ref{thm_dimmatch}, so let $\e > 0$ be given.  As a notational convenience, 
put $\Lambda = (\lambda_\a)_{k=1}^M$ and for $k > M$ we write $\lambda_\a = 0$.  Using \eqref{eq_lipderiv} again and the above identities, it follows that, as $y \to x_0$,
\begin{equation*}
\begin{split}
&
\frac{\left|u(y) - u(x_0) - 
\Lambda \cdot [{\bf f}(y)- {\bf f}(x_0)] 
\right|}{d(x_0,y)} 
\\ &
\hspace{.15in} \;=\;
\frac{\bigg|u(y) - u(x_0) - \displaystyle{
\sum_{\a=1}^\infty \lambda_\a \, [g^\a(y)-g^\a(x_0)] 
} \bigg|}{d(x_0,y)} 
\\ &
\hspace{.15in} \;\leq\;
\frac{\big| [u - v_l](y) - [u - v_l](x_0)\big|}{d(x_0,y)} \;+\;
\frac{\bigg|v_l(y) - v_l(x_0) - \displaystyle{
\sum_{\a=1}^\infty \lambda_\a [g^\a(y)-g^\a(x_0)] 
} \bigg|}{d(x_0,y)} \\ &
\hspace{.15in} \;\leq\;
\e \,+\, \Lip[u - v_l](x_0) \,+\, \Lip\bigg[\sum_{\a=1}^\infty (\lambda_{\a l} - \lambda_\a) g^\a\bigg](x_0)
\end{split}
\end{equation*}
holds for sufficiently large $l$, and thus
\begin{equation*}
\begin{split}
&
\frac{\left|u(y) - u(x_0) - 
\Lambda \cdot [{\bf f}(y)- {\bf f}(x_0)] 
\right|}{d(x_0,y)} 
\\ &
\hspace{1.15in} \;\leq\;
\e \,+\, \Lip[u - v_l](x_0) \,+\, \Lip\bigg[\sum_{\a=1}^\infty (\lambda_{\a l} - \lambda_\a) g^\a\bigg](x_0) \\ &
\hspace{1.15in} \;\leq\;
\e \,+\, 
K\sum_{i=1}^M \left|\d_i[u - v_l](x_0)\right| \,+\, 
K|\lambda_{il} - \lambda_i| \, |\d_if^i(x_0)| \\ &
\hspace{1.15in} \;\leq\;
\e \;+\; K\e \,+\, K \, L(f^i) \e.
\end{split}
\end{equation*}
Since $\e$ was arbitrary, a measurable differentiable structure exists with coordinates ${\bf f}$, and where the differential of $\varphi$ is 
$$
\hspace{.8in}
Du(x_0) \;=\; \Lambda \;=\;
(\d_1f^1(x_0))^{-1}{\bf d}u(x_0).
\hspace{.8in}
\qedhere
$$
\end{proof}

\section{The Necessity of Lip-Derivation Inequalities} \label{sect_suff}

To prove the $(1) \Rightarrow (2)$ direction of Theorem \ref{thm_doublingdiff}, we first check the validity of \eqref{eq_lipderiv} with differentials $D_mf$ which, a priori, are not known to be derivations.  It will be shown afterwards that the components of $f \mapsto D_mf$ are in fact weakly continuous (and hence are well-defined derivations).

\begin{lemma} \label{lemma_lipdiff}
If $(X,d,\mu)$ supports a measurable differentiable structure, then there is an atlas of charts $\{ (X_m,\xi_m) \}_{m=1}^\infty$ on $X$ with the following property:\ for each $m \in \N$, there exists $C > 0$ so that  for $\mu$-a.e.\ $x \in X_m$,
$$
|D_mf(x)| \;\leq\; C \Lip[f](x).
$$
\end{lemma}

As shown in \S\ref{sect_intro}, the opposite inequality already holds. 
The rest of the argument follows the proof of \cite[Thm 4.38(ii)]{Cheeger}, which we include for completeness.  

\begin{proof}
Let $\{(X_m,\xi_m)\}_{m=1}^\infty$ be an atlas of charts associated to 
$X$ and let $n = n(m)$ be the (chart) dimension of $X_m$.  Define $Y_m \subset X_m$ as the collection of points $x$ where there exists a nonzero $c = (c_1, \cdots, c_n) \in \R^n$ so that 
$$
\Lip[ c \cdot \xi_m ](x) \;=\; 0
$$
where $c \cdot \xi_m \;:=\; \sum_{i=1}^n c_i\xi_m^i$, for short.

If $\mu(Y_m) > 0$, then one component, say $\xi_m^1$, is a linear combination of the remaining ones $\{\xi_m^i\}_{i=2}^n$, with $c_1 \neq 0$.  So for $f = \xi_m^1$, both of the vectorfields
$$
x \,\mapsto\, (1, 0, \cdots, 0) \; \text{ and } \; 
x \,\mapsto\, \Big(0, -\frac{c_2}{c_1}, \cdots, -\frac{c_m}{c_1}\Big),
$$
satisfy the role of $D_mf$ in \eqref{eq_diffprop}, which contradicts the uniqueness of the differential on $X_m$.  It follows that $\mu(Y_m) = 0$.

For fixed $x \in X_m \setminus Y_m$, observe that $f \to \Lip[f](x)$ is a semi-norm, so 
$$
l_x(c) \;:=\; \Lip[c \cdot \xi](x)
$$
is a positive, continuous function on $\R^n \setminus \{0\}$.  It follows that, for $\mu$-a.e.\ $x \in X_m$,
$$
K(x) \;:=\; \min_{\mathbb{S}^{n-1}} l_x \;>\; 0.
$$

Now let $f \in \Lip(X)$ be arbitrary.  If $|D_mf(x)| = 0$, then by differentiability \eqref{eq_diffprop} and the above seminorm property, we would have
$$
\Lip[f](x) \;\leq\;
\Lip[f - D_mf(x) \cdot \xi_m](x) \,+\, \Lip[ D_mf(x) \cdot \xi_m](x) \;=\; 0 + 0
$$ 
which proves the lemma.  On the other hand, for $c = |D_mf(x)|^{-1}D_mf(x)$ note that
$$
\Lip[f](x) \;=\; 
\Lip[ Df(x) \cdot \xi ](x) \;\geq\;
K(x) |Df(x)|.
$$
The lemma follows, by partitioning charts into sub-charts of the form
$$
X_{m,k} \;:=\; 
\left\{x \in X_m \,:\, 2^{-(k+1)} \leq  K(x) < 2^{-k} \right\}\
$$
and choosing $C = 2^{-(K+1)}$.
\end{proof}

We now show that the components of the differential 
are weakly continuous in the sense of Definition \ref{defn_deriv}.  This step requires Sobolev space techniques.

In general, a doubling metric space need not possess rich families of rectifiable curves.  So instead of the Newtonian-Sobolev spaces \cite{Shan}, we will use Sobolev spaces defined in terms of measurable differentiable structures as well as the Haj{\l}asz-Sobolev spaces of functions \cite{Hajlasz}. For a further discussion of the latter function space, see also \cite{Heinonen:Koskela}, \cite{Shan}, \cite{Hajlasz:Koskela2}, \cite{Hajlasz2}, \cite{HeinonenLA}, and \cite{HeinonenNC}.

To fix notation, for a measurable differentiable structure on $X$, let $N \in \N$ be the dimension bound as in Definition \ref{defn_measdiff}.  Moreover, for a fixed atlas $\{(X_m,\xi_m)\}_{m=1}^\infty$, define a {\em global} differential map $Df : X \to \R^N$ by
$$
Df \;:=\; \sum_{m=1}^\infty \chi_{X_m} \, D_mf.
$$

\begin{defn} \label{defn_metricsobolev}
Let $(X,d,\mu)$ be a metric measure space and let $p \in (1,\infty)$.
\begin{enumerate}
\item A function $u \in L^p(X)$ lies in the {\em Haj{\l}asz-Sobolev space} $M^{1,p}(X)$ if there exists a non-negative $g \in L^p(X)$ so that
\begin{equation} \label{eq_ptwise}
|u(x) - u(y)| \;\leq\; \big( g(x) + g(y) \big) \, d(x,y) 
\, \text{ for} \, \mu\textrm{-a.e.\ } x,y \in X;
\end{equation}
\item if $X$ supports a non-degenerate measurable differentiable structure with atlas $\{(X_m,\xi_m)\}_{m=1}^\infty$, then for the linear subspace
$$
\tilde{H}^{1,p}(X) \;:=\; \{ f \in \Lip_{\rm loc}(X) \cap L^p(X) \,:\, |Df| \in L^p(X) \}
$$
of $\Lip(X)$, we define a norm by
$$
\|f\|_{1,p} \;:=\; \|f\|_{L^p(X)} \,+\, \| |Df| \|_{L^p(X)}
$$
and the space $H^{1,p}(X)$ is the completion of $\tilde{H}^{1,p}(X)$ with respect to $\|\cdot\|_{1,p}$.
\end{enumerate}
\end{defn}

Following \cite{Hajlasz}, $M^{1,p}(X)$ is a Banach space with respect to the norm
$$
\|u\|_{M^{1,p}(X)} \;:=\; 
\|u\|_{L^p(X)} \,+\, 
\inf\Big\{ \|g\|_{L^p(X)} \,;\, g : X \to [0,\infty] \text{ satisfies } \eqref{eq_ptwise} \Big\}
$$
and for $p > 1$, the infimum is always attained by some $g_u \in L^p(X)$.

\begin{rmk}
The set inclusion 
$M^{1,p}(X) \subseteq H^{1,p}(X)$
always holds on metric spaces with doubling measures \cite[Thm 9]{Franchi:Hajlasz:Koskela} and 
\begin{equation} \label{eq_normbd}
\|u\|_{1,p} \;\leq\; C \|u\|_{M^{1,p}(X)}
\end{equation}
follows, with $C > 0$ independent of $u$, from estimates involving $u \in M^{1,p}(X)$ and its Lipschitz approximations $\{u_\e\}_{\e > 0}$, as constructed from fixed partitions of unity on the space \cite[Lem 12]{Franchi:Hajlasz:Koskela}.
\end{rmk}

We now study the Banach space structures of $M^{1,p}(X)$ and $H^{1,p}(X)$.  The following result is essentially \cite[Thm 4.38(ii)]{Cheeger}, but we include the details.

\begin{lemma} \label{lemma_reflexive}
If $(X,d,\mu)$ supports a non-degenerate measurable differentiable structure, then $H^{1,p}(X)$ and $M^{1,p}(X)$ are reflexive Banach spaces, for all $1 < p < \infty$.
\end{lemma}

\begin{proof}
Choose an atlas $\{(X_m,\xi_m)\}_{m=1}^\infty$ on $X$ as in Lemma \ref{lemma_lipdiff}, and put $n = n(m)$. For $\mu$-a.e.\ $x \in X_m$, define a norm on $\R^n$ by 
$$
| v |_x \;:=\; \Lip[v \cdot \xi_m](x).
$$
Verily, \eqref{eq_lipderiv} implies that at such points $x$, we have $|v|_x = 0$ if and only if $v = 0$.

A theorem of F.\ John \cite{John}, however, asserts that every norm on $\R^n$, including $|\cdot|_x$, is comparable to the usual inner product norm $| \cdot |$ on $\R^n$ and where the multiplicative constants depend only on $n$.  This implies that $|\cdot|_x$ is a {\em uniformly convex} norm on $\R^n$ for $\mu$-a.e.\ $x \in X_m$, as well as
$$
|Df(x)| \;\approx\; \Lip[f](x) \;=\; \Lip[Df(x) \cdot \xi](x) \;=\; |Df(x)|_x.
$$
So for $p \in (1,\infty)$, the space $\tilde{H}^{1,p}(X)$ is uniformly convex \cite[Rmk 10.1.10]{handbook}, from which the uniform convexity and reflexivity of $H^{1,p}(X)$ follows \cite[\S26.6]{Kothe}.

As for the Haj{\l}asz-Sobolev space, equation \eqref{eq_normbd} implies that the inclusion map $M^{1,p}(X) \hookrightarrow H^{1,p}(X)$ is continuous.  By the Closed Graph Theorem, $M^{1,p}(X)$ is a closed subspace of $H^{1,p}(X)$, so $M^{1,p}(X)$ is also reflexive.
\end{proof}

The next result relates weak convergence in $\Lip_b(X)$, $M^{1,p}(X)$, and $H^{1,p}(X)$.

\begin{lemma}
Let $(X,d,\mu)$ support a non-degenerate measurable differentiable structure.  If $f_k \wsto f$ in $\Lip(X)$, then for every $p \in (1,\infty)$ and every ball $B$ in $X$, the sequence $\{f_k\}_k$ converges weakly to $f$ in both $M^{1,p}(B)$ and $H^{1,p}(B)$.
\end{lemma}

\begin{proof}
By duality, the reverse inclusion $[H^{1,p}(B)]^* \subset [M^{1,p}(B)]^*$ holds, so it suffices to show weak convergence in $M^{1,p}(B)$ only.

Indeed, $\{f_k\}_k$ is uniformly Lipschitz and thus bounded in $M^{1,p}(B)$ for all $p \in (1,\infty)$.  Combining Banach-Alaoglu with Lemma \ref{lemma_reflexive}, there is a subsequence $\{f_{k_j}\}_j$ that converges weakly to some $ h \in M^{1,p}(B)$.  By Mazur's lemma, a sequence of convex combinations $\{\f_j\}_j$ of $\{f_{k_j}\}_j$ converge in norm to $h$, so a further subsequence $h_i := \f_{j_i}$ converge pointwise to $f$.

On the other hand, since $\{f_k\}_{k=1}^\infty$ is uniformly Lipschitz, the convergence $f_k \to f$ is locally uniform.  The operations of taking subsequences and convex combinations therefore preserve this locally uniform convergence, so $h = f$.  In particular, this shows that every subsequence of $\{f_k\}_k$ has a further subsequence which converges weakly to $f$, so equivalently $f_k \rightharpoonup f$ in $M^{1,p}(X)$.
\end{proof}

We now prove the remaining direction of Theorem \ref{thm_doublingdiff}.  The argument is very similar to the proof in \cite{Gong} regarding the Cheeger-Weaver theorem \cite{WeaverED}.

\begin{proof}[Proof of $(1) \Rightarrow (2)$ for Theorem \ref{thm_doublingdiff}]
It remains to show, on each chart $X_m$ of $X$, that each component of the differential $f \mapsto D_mf$ is a derivation.
To simplify notation, we write $X = X_m$ and $n = n(m)$ and
$$
Df \;:=\; 
D_mf \;:=\; 
( \partial_1f, \cdots, \partial_nf ).
$$
Moreover assume that $X=X_m$ is a bounded metric space.

Fix $p > 1$ and put $q = \frac{p}{p-1}> 1$. For each $\varphi \in L^q(X)$, it follows from Lemma \ref{lemma_lipdiff} that for each $i = 1, \ldots n$, the map
\begin{equation} \label{eq_fnl}
T_i(f) \;:=\;
\int_X  \varphi \, \partial_if \,d\mu
\end{equation}
is a bounded linear functional on the vector subspace $\tilde{H}^{1,p}(X)$.  Applying Hahn-Banach, it extends to an element in $[H^{1,p}(X)]^*$, which we also denote by $T_i$.

To complete the proof, assume $f_k \to f$ pointwise and that 
$l := \sup_k L(f_k) < \infty.$
Without loss, $l > 0$; otherwise each $f_k$ is constant, so $f$ is constant and trivially 
$$
\partial_if_k \;=\; 0 \;=\; \partial_if
$$
holds for each $i = 1, \ldots n$, which would give the theorem.

Let $\psi \in L^1(X)$ and $\e > 0$ be arbitrary.  Since $X$ is bounded and $\mu$ is doubling (hence Radon), $L^q(X)$ is dense in $L^1(X)$, so there exists $\varphi \in L^q(X)$ satisfying 
$$
\|\psi - \varphi\|_{L^1(X)} \;<\; \frac{\e}{4l}.
$$
From $T_i \in [H^{1,p}(X)]^*$ and Lemma \ref{lemma_reflexive} it follows that, for sufficiently large $k \in \N$,
$$
|T_i(f_k-f)| \;=\;
\left|
\int_X \varphi \, \partial_i[f_k-f] \,d\mu 
\right| \;<\; \frac{\e}{2}.
$$
Applying the previous estimates, we further obtain
\begin{eqnarray*}
\left|
\int_X \psi \, \partial_i[f_k-f] \,d\mu 
\right| &\leq&
\left|
\int_X \varphi \, \partial_i[f_k-f] \,d\mu 
\right|
+ \| D(f_k-f) \|_{L^\infty(X)} \|\psi-\varphi\|_{L^1(X)} \\ &<&
\frac{\e}{2} + 2 L(f_k-f) \cdot \frac{\e}{4l} \;=\; \e.
\end{eqnarray*}
Since $\e$ and $\psi$ were arbitrary, it follows that $\partial_if_k \wsto \partial_if$ in $L^\infty(X)$, as desired.
\end{proof}

\bibliographystyle{alpha}
\bibliography{diffstruct}
\end{document}